\documentclass[]{article}
\makeatletter\if@twocolumn\PassOptionsToPackage{switch}{lineno}\else\fi\makeatother

\usepackage{amsfonts,amssymb,amsbsy,latexsym,amsmath,tabulary,graphicx,times,caption,fancyhdr}
\usepackage[utf8]{inputenc}
\usepackage{url,multirow,morefloats,floatflt,cancel,tfrupee}
\makeatletter

\AtBeginDocument{\@ifpackageloaded{textcomp}{}{\usepackage{textcomp}}}
\makeatother
\usepackage{colortbl}
\usepackage{xcolor}
\usepackage{pifont}
\usepackage[nointegrals]{wasysym}
\usepackage{graphicx}
\usepackage{amssymb}
\usepackage{latexsym}
\usepackage{amssymb}
\usepackage{here}
\usepackage{wrapfig}
\usepackage{fancybox}
\usepackage{amsmath}
\usepackage{multicol}
\usepackage{mathtools}
\usepackage{amsthm}
\usepackage{lineno}
\newcommand*\patchAmsMathEnvironmentForLineno[1]{
  \expandafter\let\csname old#1\expandafter\endcsname\csname #1\endcsname
  \expandafter\let\csname oldend#1\expandafter\endcsname\csname end#1\endcsname
  \renewenvironment{#1}
     {\linenomath\csname old#1\endcsname}
     {\csname oldend#1\endcsname\endlinenomath}}
\newcommand*\patchBothAmsMathEnvironmentsForLineno[1]{
  \patchAmsMathEnvironmentForLineno{#1}
  \patchAmsMathEnvironmentForLineno{#1*}}
\AtBeginDocument{
\patchBothAmsMathEnvironmentsForLineno{equation}
\patchBothAmsMathEnvironmentsForLineno{align}
\patchBothAmsMathEnvironmentsForLineno{flalign}
\patchBothAmsMathEnvironmentsForLineno{alignat}
\patchBothAmsMathEnvironmentsForLineno{gather}
\patchBothAmsMathEnvironmentsForLineno{multline}
}

\newtheorem{theo}{Theorem}[section]
\newtheorem{lem}[theo]{Lemma}
\newtheorem{rem}[theo]{Remark}
\newtheorem{coro}[theo]{Corollary}
\newtheorem{defi}[theo]{Definition}

\makeatletter

\def\mcWidth#1{\csname TY@F#1\endcsname+\tabcolsep}

\def\cAlignHack{\rightskip\@flushglue\leftskip\@flushglue\parindent\z@\parfillskip\z@skip}
\def\rAlignHack{\rightskip\z@skip\leftskip\@flushglue \parindent\z@\parfillskip\z@skip}

\@ifundefined{etal}{}{}

\usepackage{ifxetex}
\ifxetex\else\if@twocolumn\@ifpackageloaded{stfloats}{}{\usepackage{dblfloatfix}}\fi\fi

\AtBeginDocument{
\expandafter\ifx\csname eqalign\endcsname\relax
\def\eqalign#1{\null\vcenter{\def\\{\cr}\openup\jot\m@th
  \ialign{\strut$\displaystyle{##}$\hfil&$\displaystyle{{}##}$\hfil
      \crcr#1\crcr}}\,}
\fi
}

\AtBeginDocument{%
  \@ifpackageloaded{endfloat}%
   {\renewcommand\efloat@iwrite[1]{\immediate\expandafter\protected@write\csname efloat@post#1\endcsname{}}}{\newif\ifefloat@tables}%
}%

\def\BreakURLText#1{\@tfor\brk@tempa:=#1\do{\brk@tempa\hskip0pt}}
\let\lt=<
\let\gt=>
\def\processVert{\ifmmode|\else\textbar\fi}

\@ifundefined{subparagraph}{
\def\subparagraph{\@startsection{paragraph}{5}{2\parindent}{0ex plus 0.1ex minus 0.1ex}%
{0ex}{\normalfont\small\itshape}}%
}{}

\newcommand\role[1]{\unskip}
\newcommand\aucollab[1]{\unskip}
  
\@ifundefined{tsGraphicsScaleX}{\gdef\tsGraphicsScaleX{1}}{}
\@ifundefined{tsGraphicsScaleY}{\gdef\tsGraphicsScaleY{.9}}{}
\def\checkGraphicsWidth{\ifdim\Gin@nat@width>\linewidth
	\tsGraphicsScaleX\linewidth\else\Gin@nat@width\fi}

\def\checkGraphicsHeight{\ifdim\Gin@nat@height>.9\textheight
	\tsGraphicsScaleY\textheight\else\Gin@nat@height\fi}

\def\fixFloatSize#1{}
\let\ts@includegraphics\includegraphics

\def\inlinegraphic[#1]#2{{\edef\@tempa{#1}\edef\baseline@shift{\ifx\@tempa\@empty0\else#1\fi}\edef\tempZ{\the\numexpr(\numexpr(\baseline@shift*\f@size/100))}\protect\raisebox{\tempZ pt}{\ts@includegraphics{#2}}}}

\AtBeginDocument{\def\includegraphics{\@ifnextchar[{\ts@includegraphics}{\ts@includegraphics[width=\checkGraphicsWidth,height=\checkGraphicsHeight,keepaspectratio]}}}

\DeclareMathAlphabet{\mathpzc}{OT1}{pzc}{m}{it}

\def\URL#1#2{\@ifundefined{href}{#2}{\href{#1}{#2}}}

\def\UrlOrds{\do\*\do\-\do\~\do\'\do\"\do\-}%
\g@addto@macro{\UrlBreaks}{\UrlOrds}

\edef\fntEncoding{\f@encoding}

\makeatother

\newif\ifmultipleabstract\multipleabstractfalse%
\makeatletter

\def\wileyIndent{1pt}
\usepackage[paperheight=10in,paperwidth=6.5in,margin=2cm,headsep=.5cm,top=2.5cm]{geometry}

\renewenvironment{abstract}
{\vspace*{-1pc}\trivlist\item[]\leftskip\wileyIndent\hrulefill\par\vskip4pt\noindent\textbf{\abstractname}\mbox{\null}\\}{\par\noindent\hrulefill\endtrivlist}

\usepackage[]{footmisc}

\def\author#1{\gdef\@author{\hskip-\dimexpr(\tabcolsep)\hskip\wileyIndent\parbox{\dimexpr\textwidth-\wileyIndent}{\centering\bfseries#1}}}

\def\title#1{\linespread{1}\gdef\@title{\centering\bfseries\ifx\@articleType\@empty\else\@articleType\\\fi#1}}

\let\@articleType\@empty \def\articletype#1{\gdef\@articleType{{\normalfont\itshape#1}}}

\AtBeginDocument{\fancypagestyle{headings}{\fancyhf{}\fancyhead[C]{\RunningHead}\fancyhead[R]{\thepage}}\pagestyle{headings}}

 \def\audegree#1{}

\date{}

\makeatother

\usepackage[T1]{fontenc}
\makeatother
\usepackage[numbers,sort&compress]{natbib}

\def\thanksspace{{\phantom{\textsuperscript{\thefootnote}}}}
\begin{document}

\title{Rigorous numerical enclosures for positive solutions of Lane--Emden's equation with sub-square exponents}
\author{Kazuaki~Tanaka\textsuperscript{1}\thanks{Corresponding author.}\thanksspace \thanks{E-mail:                     
                     tanaka@ims.sci.waseda.ac.jp}{\thanksspace}, Michael~Plum\textsuperscript{2}, Kouta~Sekine\textsuperscript{3}, Masahide~Kashiwagi\textsuperscript{4}\space and Shin'ichi~Oishi\textsuperscript{4}~\\[-3pt]\normalsize\normalfont  \itshape ~\\
\textsuperscript{1}{Institute for Mathematical Science\unskip, Waseda University\unskip, 3-4-1 Okubo, Shinjuku-ku, Tokyo 169- 8555, Japan}~\\
\textsuperscript{2}{Institut f{\"{u}}r Analysis\unskip,  Karlsruhe Institut f{\"{u}}r Technologie\unskip, Englerstra{\ss}e 2, 76131 Karlsruhe, Germany}~\\
\textsuperscript{3}{Department of Information and Communication Systems Engineering, Chiba Institute of Technology, 2-17-1 Tsudanuma, Narashino-shi, Chiba 275-0016, Japan}~\\
\textsuperscript{4}{Faculty of Science and Engineering\unskip, Waseda University\unskip, 3-4-1 Okubo, Shinjuku-ku, Tokyo 169- 8555, Japan}}

\def\RunningHead{}\def\RunningAuthor{Tanaka \MakeLowercase{\textit{et al.}} }

\maketitle

\begin{abstract}
	The purpose of this paper is to obtain rigorous numerical enclosures for solutions of Lane--Emden's equation $-\Delta u=|u|^{p-1} u$ with homogeneous Dirichlet boundary conditions.
	We prove the existence of a nondegenerate solution $u$ nearby a numerically computed approximation $\hat{u}$ together with an explicit error bound, i.e., a bound for the difference between $ u $ and $\hat{u}$.
	In particular, we focus on the sub-square case in which $1<p<2$ so that the derivative $p|u|^{p-1}$ of the nonlinearity $|u|^{p-1} u$ is not Lipschitz continuous.
	In this case, 
	it is problematic to apply the classical Newton-Kantorovich theorem for obtaining the existence proof,
	and moreover several difficulties arise in the procedures to obtain numerical integrations rigorously.
	We design a method for enclosing the required integrations explicitly, proving the existence of a desired solution based on a generalized Newton-Kantorovich theorem.
	A numerical example is presented where an explicit solution-enclosure is obtained for $ p=3/2 $ on the unit square domain $\Omega=(0,1)^2$.

\def\keywordstitle{Keywords}
\end{abstract}
    
\section{Introduction}\label{intro}
Over the last several decades, active studies have been conducted on solutions of the Dirichlet problem for Lane-Emden's equation:
\begin{align}
	\left\{\begin{array}{l l}
		-\Delta u=f(u):=\left|u\right|^{p-1}u &\mathrm{in}\ \Omega,\\
		u=0 &\mathrm{on}\ \partial\Omega
	\end{array}\right.\label{absproblem}
\end{align}
where $\Omega\subset \mathbb{R}^{n}$ ($n=2,3,\cdots$) is a bounded domain
and $ p $ is a subcritical exponent that satisfies $ 1<p<p^* $;
$p^*=\infty$ when $n=2$ and $p^*=(n+2)/(n-2)$ when $ n\geq3 $.
Here we are interested in proving the existence of a positive solution of problem \eqref{absproblem}.
Throughout this paper, we denote the $k$th order $L^{2}$-Sobolev space over $\Omega$ by $H^{k}\left(\Omega\right)$.
We define $ H_{0}^{1}\left(\Omega\right):=\{u\in H^{1}\left(\Omega\right)\ :\ u=0\ \mathrm{on}\ \partial\Omega$ in the trace sense$\}$ with inner product $\left(\cdot,\cdot\right)_{H_{0}^{1}}:=\left(\nabla\cdot,\nabla\cdot\right)_{L^{2}}$ and norm $\left\|\cdot\right\|_{H_{0}^{1}}:=\left\|\nabla\cdot\right\|_{L^{2}}$.
Moreover, $H^{-1}$ denotes the dual space of $ H^1_0(\Omega) $ with the usual supremum norm.

Positive solutions of problem \eqref{absproblem} have been analytically investigated from various points of view --- uniqueness, multiplicity, nondegeneracy, symmetry, and so on \cite{lions1982existence,gidas1979symmetry,lin1994uniqueness,damascelli1999qualitative,gladiali2011bifurcation,de2019morse}.
On the other hand, several investigations have been conducted to obtain existence proofs via rigorous numerical enclosures for solutions of problem \eqref{absproblem} \cite{plum2001computer,mckenna2009uniqueness,mckenna2012computer,takayasu2014remarks,pacella2017computer,tanaka2017sharp,tanaka2020numerical}.
Such approaches are known as computer-assisted proofs, numerical verification, validated numerics, or verified numerical computations and have been applied to various problems, including some for which purely analytical methods have failed (see, for example, \cite{day2007validated,plum2008,gameiro2010analytic,van2010global,arioli2010computer,nakao2011numerical,gameiro2011rigorous,arioli2012non}).
In the survey book \cite{nakaoplumwatanabe2019numerical},
these topics are summarized and extended.
In \cite{plum2001computer}, a nontrivial solution on the unit square $ \Omega = (0,1)^2 $ was enclosed for $ p=2 $ using the quadratic triangular finite element method.
For the same domain,  existence and -- more important -- uniqueness was proved with explicit error bounds when $ p=2 $ \cite{mckenna2009uniqueness} and $ p=3 $ \cite{mckenna2012computer}, even when a linear term is added.
In \cite{takayasu2014remarks}, for the L-shaped domain $ \Omega=(-1,1)^2 \backslash [-1,0]^2 $ where a singularity arises at the re-entrant corner, a nontrivial solution was enclosed for $ p=2 $ via the classical Newton-Kantorovich's theorem close to an approximate solution constructed by a quadratic finite element basis on a non-uniform triangulation.
In \cite{pacella2017computer},
the existence of a non-trivial solution to Lane-Emden's equation was proved on an unbounded L-shaped domain, where purely analytical methods were unable to give existence.
Subsequently, in \cite{tanaka2017sharp}, nontrivial solutions on the unit square were very sharply enclosed for integers $ p=2,3,4,5,6 $ and applied to inclusions of the best constant for Sobolev embeddings.
In the recent study \cite{tanaka2020numerical}, methods for proving the positivity of enclosed solutions of elliptic problems were proposed and applied to problem \eqref{absproblem} with $ p=3,5 $.
Despite such successes for integers $ p $,
no explicit solution enclosure has been shown when $ p $ is a noninteger.


From this background, we pay attention to problem \eqref{absproblem} with $p \in (1,2)$.
Among the cases in which $p$ is noninteger, this case is especially challenging because the derivative $p|u|^{p-1}$ of the nonlinearity $|u|^{p-1} u$ is not  Lipschitz continuous but singular at the points where $u=0$.
Moreover, the amplitude of a solution of \eqref{absproblem} is much larger than when $p \geq 2$, raising as $ p \rightarrow 1 $.
These features give rise to some difficulties in rigorous computations for \eqref{absproblem}.

Our purpose is to prove the existence of a nondegenerate positive solution of \eqref{absproblem} nearby a numerically computed approximate solution $ \hat{u} \in H^1_0(\Omega) \cap L^{\infty}(\Omega) $ together with an explicit error bound in terms of the norm $\left\|\cdot\right\|_{H^1_0\left(\Omega\right)}$.
The existence proof is based on Theorem \ref{theo:plum2001} below,
an improved version of Newton-Kantorovich's theorem \cite{nakaoplumwatanabe2019numerical,plum2001computer,plum2008}.
This theorem was derived by applying a fixed point argument to 
Newton's operator for the error $u-\hat{u}$ (see Remark \ref{rem:Theorem1proof} for more information).
Because the condition on the Lipschitz continuity for $ f' $ is relaxed from the original version, Theorem \ref{theo:plum2001} is well applicable to even the case in which $p \in (1,2)$.
After the existence proof in terms of an enclosure in the $ H^1_0 $-norm, we further obtain a ball including the solution in terms of the norm $\left\|\cdot\right\|_{L^{\infty}\left(\Omega\right)}$ in the regular case $ u \in H^1_0(\Omega) \cap H^2(\Omega) $ by considering a norm bound for the embedding $ H^2(\Omega) \hookrightarrow L^{\infty}(\Omega) $.


Rigorously enclosing integrals is required at many points in the above-mentioned process, e.g., when estimating the norm of the residual norm $\left\| \Delta \hat{u} + f(\hat{u})  \right\|_{H^{-1}}$ and when computing the operator norm bound of the inverse of the linearized operator $-\Delta - f'(\hat{u}): H^1_0(\Omega) \rightarrow H^{-1}$. 
The integrands contain $f'\left(\hat{u}\right)$ or $f\left(\hat{u}\right)$,
and the smoothness of derivatives of $f$ is useful for avoiding technical difficulties in computing the integrals. 
However, in our situation where $f(u)=|u|^{p-1}u~\left(1<p<2\right)$, such integrations are difficult to calculate with high-precision as well as rigorous estimates because the second derivative of $f\left(v \right)$ is not bounded around the points at which $v = 0$ even for smooth functions $v$.
One of this paper's main contributions is to develop enclosure methods for such irregular integrations required in the process of existence proofs.



We quote
\cite{watanabe1997verified}
where 
a rigorous enclosure was obtained for a solution of $ -\Delta u = \lambda \max\{0,u\} $ ($\lambda \geq 0$) subject to the boundary condition $ u=-1 $.
Although the nonlinearity $t \in \mathbb{R} \mapsto \lambda \max\{0,t\} \in \mathbb{R}$ is not differentiable,
the operator $u \in H^1_0(\Omega) \mapsto \lambda \max\{0,u\} \in H^{-1}$ is Fr\'echet differentiable at functions that are non-zero almost everywhere.
The enclosure was obtained by a method called {\it FN-Int} 
based on a fixed-point argument for an infinite-dimensional Newton-type operator, which is split into a finite-dimensional part, treated by verifying numerical linear algebra, and an infinite-dimensional remainder, which is captured by a projection error bound (see \cite{nakaoplumwatanabe2019numerical} for details).
Because the effectiveness of both Theorem \ref{theo:plum2001} and FN-Int depends on the evaluation of several constants and computational implementations, it is very tough to compare them at a general level.
However, the dominant feature of Theorem \ref{theo:plum2001} is that it directly handles the infinite-dimensional Newton's operators without splitting (see Remark \ref{rem:Theorem1proof}) and can be applied to unbounded domains.



The rest of this paper is organized as follows.
Section \ref{verificationtheory} shows methods to obtain solution enclosures for problem \eqref{absproblem}.
Section \ref{integration} discusses rigorous numerical integration methods required for implementing actual computation to obtain the enclosures.
In Section \ref{sec/example}, we present a numerical example where a positive solution of \eqref{absproblem} is enclosed when $ p=3/2 $ and $ \Omega=(0,1)^2 $.

\section{Enclosure methods for elliptic problems}\label{verificationtheory}
We begin by introducing the required notation.
We denote $V=H_{0}^{1}\left(\Omega\right)$, and $V^{*}=H^{-1}$.
The $L^{2}$-inner product and norm are simply denoted by $\left(\cdot,\cdot\right)$ and $\|\cdot\|$, respectively, if no confusion arises.
For two Banach spaces $X$ and $Y$,
the space of bounded linear operators from $X$ to $Y$ is denoted by $\mathcal{B} (X,Y)$.
The norm of $T \in \mathcal{B} (X,Y)$ is defined by 
\begin{align}
	\label{eq:normofdual}
	\|T\|_{\mathcal{B}(X, Y)}:=\sup _{0 \neq u \in X} \frac{\|T u\|_{Y}}{\|u\|_{X}}.
\end{align}
We define the operator $\mathcal{F}:V\rightarrow V^{*}$ by $\mathcal{F}(u):=-\Delta u- |u|^{p-1} u$, which satisfies
\begin{align}
	\left\langle\mathcal{F}(u),v\right\rangle=\left(\nabla u,\nabla v\right)-\left(|u|^{p-1} u,v\right)~~~{\rm for}~u,v\in V.
\end{align}
The Fr\'echet derivative $ {\mathcal F'_{\varphi}} $ of $ \mathcal{F} $ at $ \varphi \in V $ is given by
\begin{align}
&\langle \mathcal F'_{\varphi}u,v \rangle = \left(\nabla u,\nabla v\right) - p( |\varphi|^{p-1}u,v )  \text{~~~for~all~} u,v \in V. \label{def:derivativecalf}
\end{align}
We look for a solution $ u \in V $ of
\begin{align}
	\mathcal{F}(u)=0\label{gFproblem},
\end{align}
which is equivalent to the weak form of \eqref{absproblem}.
A norm bound for the embedding $V\hookrightarrow L^{p}\left(\Omega\right)$ is denoted by $C_{p}$, a positive number that satisfies
\begin{align}
\left\|u\right\|_{L^{p}(\Omega)}\leq C_{p}\left\|u\right\|_{V}~~~{\rm for~all}~u\in V.\label{embedding}
\end{align}
Note that $\|u\|_{H^{-1}} \leq C_{p}\|u\|_{L^{p'}}$, $u \in L^{p'}(\Omega)$, holds for $p'$ satisfying $p^{-1}+p'^{-1}=1$ (see, for example, \cite[Section 4]{plum2008}).
Calculating an explicit upper bound for $C_{p}$ is important for our enclosure method.
Corollary \ref{roughboundtheo} provides a formula that gives an upper bound given a bounded domain $\Omega$.
Another alternative formula can be found in \cite[Lemma 7.10]{nakaoplumwatanabe2019numerical}.

In the following, we discuss a method for enclosing solutions to \eqref{gFproblem} near a numerically computed approximate solution $\hat{u} \in V$ in terms of the norms $\|\cdot\|_V$ and $\|\cdot\|_{L^{\infty}}$.
Although in some place of this section, the regularity $\hat{u} \in L^{\infty}(\Omega)$ is additionally assumed,
this assumption impairs little of the flexibility of actual numerical computation methods.
Moreover, in Subsection \ref{subsecLinf}, we assume the further regularity $\Delta \hat{u} \in L^{2}(\Omega)$ to obtain an $ L^{\infty} $-error bound using a norm bound for the embedding $ H^2(\Omega) \hookrightarrow L^{\infty}(\Omega) $,
and the fact that solutions of problem \eqref{absproblem}, with $p=3/2$, are in $H^2(\Omega)$.



\subsection{$H_{0}^{1}$-error estimation}\label{subsecH10}
We use the following theorem, a generalization of the classical Newton-Kantorovich theorem, for obtaining an $H_{0}^{1}$ error estimation of solutions to \eqref{gFproblem}, given an approximation $\hat{u} \in V$.
\begin{theo}[\cite{nakaoplumwatanabe2019numerical,plum2001computer, plum2008}]\label{theo:plum2001}
	Let $\hat{u}\in V$ be some numerical approximate solution to \eqref{absproblem}.
	Suppose that some $\delta>0,\ K>0$, and a non-decreasing function $g$ are known to satisfy
	\begin{align}
	&\left\|\mathcal{F}\left(\hat{u}\right)\right\|_{V^{*}}\leq\delta,\label{zansa}\\
	&\left\|u\right\|_{V}\leq K\left\|\mathcal{F}_{\hat{u}}'u\right\|_{V^{*}}~~~{\rm for~all}~u\in V,\label{inverse}\\
	&\left\|\mathcal{F}_{\hat{u}+u}'-\mathcal{F}_{\hat{u}}'\right\|_{\mathcal{B}(V,V^{*})}\leq g\left(\left\|u\right\|_{V}\right)~~~{\rm for~all}~u\in V,\label{lip}
	\shortintertext{and}
	&g(t)\rightarrow 0~~{\rm as}~~t\rightarrow 0\label{gconv}.
	\end{align}
	Moreover, suppose that, for some $\alpha>0$,
	\begin{align}
	\displaystyle \delta\leq\frac{\alpha}{K}-G\left(\alpha\right)~~and~~Kg\left(\alpha\right)<1,\label{alphacond}
	\end{align}
	where $G(t):=\displaystyle \int_{0}^{t}g(s)ds$.
	Then, there exists a solution $u\in V$ to the equation $\mathcal{F}(u)=0$ satisfying
	\begin{align}
	\left\|u-\hat{u}\right\|_{V}\leq\alpha.\label{al}
	\end{align}
	The solution is moreover nondegenerate and unique under the side condition \eqref{al}.
\end{theo}

\begin{rem}
	In {\rm \cite[Theorem 6.2]{nakaoplumwatanabe2019numerical}} an improved uniqueness-condition is provided, which gives a larger uniqueness area.
\end{rem}

\begin{rem}
The classical Newton-Kantorovich theorem requires the Fr\'echet derivative of $\mathcal{F}$ to be Lipschitz continuous for any two points in $D \subset V$, a ball centered at $\hat{u}$.
In other words, this theorem needs a positive constant $L$ such that $\|\mathcal{F}_{u}'-\mathcal{F}_{v}'\|_{\mathcal{B}(V,V^{*})}\leq L\|u-v\|_{V}~{\rm for~all}~u,v\in D$.
However, when $ \hat{u}=0 $ in some parts of $ \Omega $, such $L$ does not exist for $p \in (1,2)$ because $D$ includes functions that vanish in subdomains of positive measure.
Therefore, it is at least very problematic to use the classical version for the sub-square case.
On the other hand, Theorem $\ref{theo:plum2001}$ with the relaxed Lipschitz-continuity condition \eqref{lip} is well applicable to this case.
\end{rem}

\begin{rem}\label{rem:Theorem1proof}
Theorem $\ref{theo:plum2001}$ was obtained by applying Banach's fixed-point theorem to the following fixed-point equation for the error $v=u-\hat{u}$
\begin{align}
	v=-\mathcal{F}_{\hat{u}}'^{-1}[\mathcal{F}(\hat{u})+\{\mathcal{F}(\hat{u}+v)-\mathcal{F}(\hat{u})-\mathcal{F}_{\hat{u}}'[v]\}]=: T(v).
\end{align}
Therefore, we can interpret that, under the required conditions \eqref{zansa}--\eqref{alphacond}, the existence of a fixed point $v$ of Newton's operator $T$ is ensured nearby $0$. 
Another version of the theorem based on Schauder’s fixed-point theorem was also proposed in {\rm \cite{nakaoplumwatanabe2019numerical,plum2008}}.
This requires the continuity and the compactness of $T$, but does not need the second inequality assumption in \eqref{alphacond}.
See {\rm\cite{nakaoplumwatanabe2019numerical}} for a more detailed discussion.
\end{rem}

By adding a further condition to Theorem \ref{theo:plum2001}, the positivity of the enclosed solution $ u $ can be confirmed.
Before describing the condition, 
we define the positive and the negative parts $v_{+},v_{-} \in V$ of a function $v \in V$ by
\begin{align*}
	v_{+}=\displaystyle \max\left\{v,0\right\}
	~~{\rm and}~~
	v_{-}=\displaystyle \max\left\{-v,0\right\},
\end{align*}
respectively.
The following result is obtained by unifying Theorem \ref{theo:plum2001} and \cite[Corollary 2.5]{tanaka2020numerical}.

\begin{coro}\label{up_positive_newkan}
	Let $\hat{u}\in V$ be some approximate solution to \eqref{absproblem}.
	Suppose that some $\delta>0,\ K>0$, and a non-decreasing function $g$ satisfying \eqref{zansa}--\eqref{gconv} are known,
	and that some $\alpha>0$ exists satisfying \eqref{alphacond}.
	If we have
	\begin{align}
		C_{p+1}^{2} (\left\|\hat{u}_{-}\right\|_{L^{p+1}}+C_{p+1}\alpha)^{p-1} < 1
		\label{uptheo_2_katei}
	\end{align}
	then there exists a positive solution $u\in V$ to \eqref{gFproblem} satisfying \eqref{al}.
	Furthermore, this solution is nondegenerate and unique under the side condition \eqref{al}.
\end{coro}
\begin{proof}
	The existence, nondegeneracy, and local uniqueness of the solution $u$ satisfying \eqref{al} is ensured by Theorem {\rm \ref{theo:plum2001}}.
	The positivity follows from \cite[Corollary 2.5]{tanaka2020numerical}.
\end{proof}

\begin{rem}
	The norm $ \left\|\hat{u}_{-}\right\|_{L^{p+1}} $ is small when $\hat{u}$ is close to a nonnegative function.
	Therefore, if a solution enclosure derived from Theorem {\rm \ref{theo:plum2001}} is sufficiently accurate for almost nonnegative $\hat{u}$,
	the positivity of $ u $ can be confirmed because the left-hand side of \eqref{uptheo_2_katei} is ``small''.
\end{rem}

In the rest of this subsection, we discuss the computation of $\delta$, $K$, and $g$ required in Theorem \ref{theo:plum2001}.
\subsubsection*{Residual bound $\delta$}
For $\hat{u}\in V$ satisfying $\Delta\hat{u}\in L^{2}\left(\Omega\right)$, the residual bound $\delta$ can be computed as
\begin{align*}
C_{2}\left\|\Delta\hat{u}+\left|\hat{u}\right|^{p-1}\hat{u}\right\|_{L^2},
\end{align*}
where $C_2$ is an embedding constant satisfying \eqref{embedding} for $p=2$.
Our numerical experiment discussed in Section \ref{sec/example} uses this evaluation, and we calculate the $L^{2}$-norm rigorously via the numerical integration method described in Section \ref{integration}.

	The condition $\Delta \hat{u} \in L^2 (\Omega)$ is not satisfied, e.g., when we construct $\hat{u}$ with a piecewise linear finite element basis.
	We use the method in \cite[Subsection 7.2]{nakaoplumwatanabe2019numerical} to evaluate the residual norm in such a case.
	The following is a brief description of the evaluation method.
	First, we find an approximation $\rho \in H(\operatorname{div}, \Omega)=\left\{\tau \in L^{2}(\Omega)^{n} : \operatorname{div} \tau \in L^{2}(\Omega)\right\}$ to $\nabla \hat{u}$.
	Then, the residual norm is evaluated as 
	\begin{align*}
		\|{F}(\hat{u})\|_{H^{-1}} &=\|-\Delta \hat{u}-\left|\hat{u}\right|^{p-1}\hat{u}\|_{H^{-1}}, \\
		&=\|-\Delta \hat{u}+\operatorname{div} \rho-\operatorname{div} \rho-\left|\hat{u}\right|^{p-1}\hat{u}\|_{H^{-1}}, \\
		& \leq\|\operatorname{div}(-\nabla \hat{u}+\rho)\|_{H^{-1}}+\|\operatorname{div} \rho+\left|\hat{u}\right|^{p-1}\hat{u}\|_{H^{-1}}, \\
		& \leq\|-\nabla \hat{u}+\rho\|_{L^2}+C_{2}\|\operatorname{div} \rho+\left|\hat{u}\right|^{p-1}\hat{u}\|_{L^2},
	\end{align*}
	where we used $\|\mathrm{div} \omega\|_{H^{-1}} \leq\|\omega\|_{L^2}$ for $\omega \in H(\mathrm{div}, \Omega)$.
	Note that $\rho$ can be computed without additional computational resources when we use mixed finite element methods to construct $\hat{u}$.

\subsubsection*{Bound $K$ for the operator norm of $\mathcal{F}_{\hat{u}}^{\prime-1}$}
We compute a bound $K$ for the operator norm of $\mathcal{F}_{\hat{u}}^{\prime-1}$ via the following theorem,
proving simultaneously that this inverse operator exists and is defined on the whole of $V^{*}$.
\begin{theo}\label{invtheo}
Let $\Phi:V\rightarrow V^{*}$ be the canonical isometric isomorphism, that is, $\Phi$ is given by
\begin{align*}
	\left\langle\Phi u,v\right\rangle:=\left(u,v\right)_{V}
	=\left(\nabla u,\nabla v\right)
	~~~{\rm for}~u,v\in V.
\end{align*}
Suppose that $\hat{u} \neq 0$ a.e. in $\Omega$.
Then, there exists an orthonormal basis $\{\psi_k\}_{k \in \mathbb{N}}$ of $V$ consisting of eigenfunctions of $\Phi^{-1}\mathcal{F}_{\hat{u}}'$,
and the associated eigenvalues $\mu_k$ form a monotonically non-decreasing sequence in $(-\infty,1)$ converging to $1$.
If
\begin{align}
	\label{mu0}
	\displaystyle \mu_{0}
	:=\min\left\{|\mu_k|\ : k \in \mathbb{N} \right\}
	>0,
\end{align}
then the inverse of $\mathcal{F}_{\hat{u}}':V \rightarrow V^*$ exists, and
\begin{align}
	\label{Ktheo}
	\left\|\mathcal{F}_{\hat{u}}^{\prime-1}\right\|_{B(V^{*},V)}\leq\mu_{0}^{-1}.
\end{align}
\end{theo}
\begin{proof}
We prove this theorem by adopting a theory of Fredholm operators; that is, we have recourse to the fact that the injectivity and the surjectivity of a Fredholm operator are equivalent.

The operator $N:=\Phi-\mathcal{F}_{\hat{u}}'$ from $V$ to $V^{*}$ is given by $\left\langle Nu,v\right\rangle=p(\left|\hat{u}\right|^{p-1}u,v)$ for all $u,v\in V$.
Thus, $N$ maps $V$ into $L^{2}(\Omega)$;
note that $p \leq 2$ and $n \leq 3$.
Hence, $N:V\rightarrow V^{*}$ is compact due to the compactness of the embedding $L^{2}(\Omega)\hookrightarrow V^{*}$.
Therefore, $\Phi^{-1} N: V \rightarrow V$ is compact, and furthermore, for $u,v \in V$,
\begin{align*}
	(\Phi^{-1}Nu,v)_V = \left\langle Nu,v \right\rangle = p (|\hat{u}|^{p-1}u,v),
\end{align*}
implying that $\Phi^{-1}N$ is $(\cdot,\cdot)_V$-symmetric and positive, since $\hat{u}\neq 0$ a.e. in $\Omega$.
Thus, there exists an orthonormal basis $\{\psi_k\}_{k \in \mathbb{N}}$ of $V$ consisting of eigenfunctions of $\Phi^{-1} N$, such that the associated eigenvalues $\nu_k$ form a monotonically non-increasing sequence in $(0,\infty)$ converging to $0$.
Because $\Phi^{-1}\mathcal{F}_{\hat{u}}' = {\rm id}_V - \Phi^{-1}N$ the first part of the assertion follows with $\mu_k:=1-\nu_k$ ($k\in \mathbb{N}$).
Moreover, the eigenfunctions of $\Phi^{-1}\mathcal{F}_{\hat{u}}'$ associated with the eigenvalues $\mu_k$ are the functions $\psi_k$ again.
Therefore, for $u \in V$,
\begin{align*}
	\left\|\mathcal{F}_{\hat{u}}'u\right\|_{V^{*}}^{2}
	=&\left\|\Phi^{-1}\mathcal{F}_{\hat{u}}' u\right\|_{V}^{2}
	=\sum_{k=1}^{\infty} \mu_k^2 |(u,\psi_k)_V|^2
	\geq\mu_{0}^{2}\left\|u\right\|_{V}^{2}.
\end{align*}
Thus, $\mathcal{F}_{\hat{u}}'$ is one-to-one, and \eqref{Ktheo} follows as soon as we have proved the surjectivity of $\mathcal{F}_{\hat{u}}'$.
Indeed, the compactness of $\Phi^{-1}N$, together with the identity $\Phi^{-1}\mathcal{F}_{\hat{u}}' = {\rm id}_V - \Phi^{-1}N$, shows that $\Phi^{-1}\mathcal{F}_{\hat{u}}'$ is a Fredholm operator.
Hence, its surjectivity, and thus surjectivity of $\mathcal{F}_{\hat{u}}'$, follows from injectivity.
\end{proof}
The eigenvalue problem $\Phi^{-1}\mathcal{F}_{\hat{u}}'u=\mu u$ in $V$ is equivalent to
\begin{align*}
\left(\nabla u,\nabla v\right)-p(\left|\hat{u}\right|^{p-1}u,v)=\mu\left(\nabla u,\nabla v\right)~~{\rm for~all}~v\in V.
\end{align*}
Setting $\lambda_k=(1-\mu_k)^{-1}$, we see that $(\lambda_{k},\psi_k)$ ($k \in \mathbb{N}$) are the eigenpairs of the problem
\begin{align}
{\rm Find}~u\in V~{\rm and}~\lambda\in \mathbb{R}~{\rm s.t.}~\left(\nabla u,\nabla v\right)=\lambda(p\left|\hat{u}\right|^{p-1}u,v)~~{\rm for~all}~v\in V.\label{eiglam}
\end{align}
To compute $K$ using Theorem \ref{invtheo},
we explicitly enclose the eigenvalue $\lambda$ of \eqref{eiglam} that minimizes the corresponding absolute value of $|\mu|\left(=|1-\lambda^{-1}|\right)$ by considering the following approximate eigenvalue problem
\begin{align}
{\rm Find}&~u\in V_{N}~{\rm and}~\lambda^{N}\in \mathbb{R}~{\rm s.t.}\nonumber\\
&\left(\nabla u_{N},\nabla v_{N}\right)=\lambda^{N}(p\left|\hat{u}\right|^{p-1}u_{N},v_{N})~~{\rm for~all}~v_{N}\in V_{N},\label{applam}
\end{align}
where $V_{N}$ is a finite-dimensional subspace of $V$.
Note that \eqref{applam} amounts to a matrix eigenvalue problem, the eigenvalues of which can be enclosed by rigorous numerical linear algebra (see, for example, \cite{behnke1991calculation,rump1999book,miyajima2012numerical}).

To estimate the error between the $k$th eigenvalue $\lambda_{k}$ of \eqref{eiglam} and the $k$th eigenvalue $\lambda_{k}^{N}$ of \eqref{applam}, we consider the weak formulation of the Poisson equation
\begin{align}
\left(\nabla u,\nabla v\right)=\left(g,v\right)~~~{\rm for~all}~v\in V,\label{poisson}
\end{align}
given $g\in L^{2}\left(\Omega\right)$.
This equation has a unique solution $u\in V$ for each $g\in L^{2}\left(\Omega\right)$.
Moreover, we introduce the orthogonal projection $P_{N}:V\rightarrow V_{N}$ defined by
\begin{align*}
\left(P_{N}u-u,v_{N}\right)_{V}=0~~~{\rm for~all}~u\in V{\rm~and~}v_{N}\in V_{N}.
\end{align*}
The following theorem enables us to estimate the error between $\lambda_{k}$ and $\lambda_{k}^{N}$.
\begin{theo}[\cite{tanaka2014verified,liu2015framework}]\label{eigtheo}
Suppose that $\hat{u} \in H^1_0(\Omega)\cap L^{\infty}(\Omega)$, and let $C_{N}$ denote a positive number such that
\begin{align}
\left\|u_{g}-P_{N}u_{g}\right\|_{V}\leq C_{N}\left\|g\right\| \label{CN}
\end{align}
for any $g\in L^{2}\left(\Omega\right)$ and the corresponding solution $u_{g}\in V$ to \eqref{poisson}.
Then,
\begin{align*}
\displaystyle \frac{\lambda_{k}^{N}}{\lambda_{k}^{N}C_{N}^{2}\|p\left|\hat{u}\right|^{p-1}\|_{L^{\infty}}+1}\leq\lambda_{k}\leq\lambda_{k}^{N}.
\end{align*}
\end{theo}
The right inequality is well known as Rayleigh-Ritz bound, which is derived from the min-max principle:
\begin{align*}
\displaystyle \lambda_{k}=\min_{H_{k}\subset V}\left(\max_{v\in H_{k}\backslash\{0\}}\frac{\left\|\nabla v\right\|^{2}}{\left\|av\right\|^{2}}\right)\leq\lambda_{k}^{N},
\end{align*}
where we set $a=\sqrt{p\left|\hat{u}\right|^{p-1}}$ and the minimum is taken over all $k$-dimensional subspaces $H_{k}$ of $V$.
Moreover, the proof of the left inequality can be found in \cite{tanaka2014verified,liu2015framework}.
Assuming the $H^{2}$-regularity of solutions to \eqref{poisson} (given, e.g., when $\Omega$ is convex {\rm \cite{grisvard2011elliptic}}), \cite[Theorem 4]{tanaka2014verified} ensures the left inequality.
A more general statement that does not require the $H^{2}$-regularity can be found in \cite[Theorem 2.1]{liu2015framework}.

\begin{rem}
When the $H^{2}$-regularity of solutions to \eqref{poisson} is confirmed a priori, e.g., when $\Omega$ is convex {\rm \cite{grisvard2011elliptic}}, \eqref{CN} can be replaced by
\begin{align}
\label{eq:CN-H2regular}
\left\|u-P_{N}u\right\|_{V}\leq C_{N}\left\|-\Delta u\right\|~~~{\rm for~all}~u\in H^{2}(\Omega)\cap V.
\end{align}
The computation of an explicit value of $C_{N}$ for a given subspace $V_{N}$ is discussed in Section $\ref{sec/example}$.
\end{rem}

\begin{rem}
	Another effective approach to calculating lower bounds for the eigenvalues $\lambda$ of \eqref{eiglam} is the homotopy-based method proposed by the second author of this paper {\rm \cite{plum1991bounds, plum1994enclosures, nakaoplumwatanabe2019numerical}}.
	This approach does not require the constant $C_N$ and thus is applicable even to unbounded domains $\Omega$.
	Moreover, the accuracy of eigenvalues can be further improved using the Temple-Lehmann-Goerisch method; see {\rm\cite{plum1994enclosures}} and {\rm \cite[Theorem 10.31]{nakaoplumwatanabe2019numerical}}.
	This refinement can be applied using the rough lower bound from Theorem {\rm\ref{eigtheo}} if more accuracy is required.
\end{rem}

\subsubsection*{Lipschitz bound-function $g$ for $\mathcal{F}_{\hat{u}}'$}
Explicitly constructing $g$ that satisfies \eqref{lip} and \eqref{gconv} in Theorem $\ref{theo:plum2001}$ is required for our enclosure method.
The following lemma is used for the construction.
\begin{lem}\label{ablem}
For $a,b\in \mathbb{R}$ and $q\in\left(0,1\right)$,
\begin{align*}
\left|\left|a+b\right|^{q}-\left|a\right|^{q}\right|\leq\left|b\right|^{q}.
\end{align*}
\end{lem}
\begin{proof}
For $\alpha,\beta\in[0,\infty)$ we have
\begin{align*}
(\displaystyle \alpha+\beta)^{q}-\alpha^{q}=q\int_{0}^{\beta}(\alpha+t)^{q-1}dt\leq q\int_{0}^{\beta}t^{q-1}dt=\beta^{q},
\end{align*}
which readily gives
\begin{align*}
\left|a+b\right|^{q}-\left|a\right|^{q}\leq\left|b\right|^{q}~~~{\rm for}~a,b\in \mathbb{R}.
\end{align*}
Redefining terms, this inequality also implies
\begin{align*}
\left|a\right|^{q}-\left|a+b\right|^{q}=\left|(a+b)+(-b)\right|^{q}-\left|a+b\right|^{q}\leq\left|-b\right|^{q}=\left|b\right|^{q}~~~{\rm for}~a,b\in \mathbb{R}
\end{align*}
and hence the assertion.
\end{proof}

The following theorem gives us an explicit function $g$ in Theorem $\ref{theo:plum2001}$ for the nonlinearity $f\left(u\right)=\left|u\right|^{p-1}u~(1<p<2)$.

\begin{theo}\label{selectiong}
	For $1<p<2$, the function
	\begin{align}
	g\left(t\right)=pC_{r}C_{s}C_{q(p-1)}^{p-1}t^{p-1}\label{g}
	\end{align}
	satisfies \eqref{lip} and \eqref{gconv} in Theorem $\ref{theo:plum2001}$, where $q,r,s$ are positive numbers that satisfy $q^{-1}+r^{-1}+s^{-1}=1$ and $q\left(p-1\right)\geq 1$.
\end{theo}
\begin{proof}
For fixed $\hat{u}\in V$, the left-hand side of \eqref{lip} is written as
\begin{align*}
\displaystyle \left\|\mathcal{F}_{\hat{u}+u}'-\mathcal{F}_{\hat{u}}'\right\|_{\mathcal{B}(V,V^{*})}=p\sup_{v,\phi\in V\backslash\{0\}}\frac{\left|\left((\left|\hat{u}+u\right|^{p-1}-\left|\hat{u}\right|^{p-1})v,\phi\right)\right|}{\left\|v\right\|_{V}\left\|\phi\right\|_{V}}.
\end{align*}
We have
\begin{align*}
\left|\left((\left|\hat{u}+u\right|^{p-1}-\left|\hat{u}\right|^{p-1})v,\phi\right)\right|&\leq\left\|\left|\hat{u}+u\right|^{p-1}-\left|\hat{u}\right|^{p-1}\right\|_{L^{q}}\left\|v\right\|_{L^{r}}\left\|\phi\right\|_{L^{s}}\\
&\leq C_{r}C_{s}\left\|\left|\hat{u}+u\right|^{p-1}-\left|\hat{u}\right|^{p-1}\right\|_{L^{q}}\left\|v\right\|_{V}\left\|\phi\right\|_{V},
\end{align*}
and Lemma \ref{ablem} ensures that
\begin{align*}
\left\|\left|\hat{u}+u\right|^{p-1}-\left|\hat{u}\right|^{p-1}\right\|_{L^{q}}&=\left(\int_{\Omega}\left|\left|\hat{u}(x)+u(x)\right|^{p-1}-\left|\hat{u}(x)\right|^{p-1}\right|^{q}dx\right)^{1/q}\\
&\leq\left(\int_{\Omega}\left|u(x)\right|^{q(p-1)}dx\right)^{1/q}=\left\|u\right\|_{L^{q(p-1)}}^{p-1}.
\end{align*}
Therefore,
\begin{align*}
\left\|\mathcal{F}_{\hat{u}+u}'-\mathcal{F}_{\hat{u}}'\right\|_{\mathcal{B}(V,V^{*})}\leq pC_{r}C_{s}C_{q(p-1)}^{p-1}\left\|u\right\|_{V}^{p-1}=g\left(\left\|u\right\|_{V}\right).
\end{align*}
\end{proof}
\subsection{$L^{\infty}$-error estimation}\label{subsecLinf}
In this subsection, we discuss a method that gives an $L^{\infty}$-error bound for a solution of \eqref{absproblem} from a known $H_{0}^{1}$ error bound.
More precisely, we compute an explicit bound for $\left\|u-\hat{u}\right\|_{L^{\infty}}$ for a solution $u\in V$ of \eqref{absproblem} satisfying
\begin{align}
\left\|u-\hat{u}\right\|_{V}\leq\alpha,\label{h10error}
\end{align}
given $\alpha>0$ and $\hat{u}\in V$.
We now assume that $\Omega \subset \mathbb{R}^2$ is convex and polygonal.
This condition gives the $H^{2}$-regularity of solutions to \eqref{absproblem} and therefore ensures their boundedness a priori (see, for example, \cite{grisvard2011elliptic}).
A solution $u$ satisfying \eqref{h10error} can be written in the form $ u=\hat{u}+\alpha\omega$ with some $\omega\in V,\ \left\|\omega\right\|_{V} \leq1$.
Moreover, $\omega$ satisfies
\begin{align*}
\left\{\begin{array}{l l}
-\Delta\alpha\omega=\left|\hat{u}+\alpha\omega\right|^{p-1}\left(\hat{u}+\alpha\omega\right)+\Delta\hat{u} &\mathrm{i}\mathrm{n}\ \Omega,\\
\omega=0 &\mathrm{on}\ \partial\Omega,
\end{array}\right.
\end{align*}
and therefore has the $H^{2}$-regularity if $\Delta\hat{u}\in L^{2}(\Omega)$.
We use the following theorem to obtain an $L^{\infty}$-error bound.
\begin{theo}[{\cite[Theorem 1]{plum1992explicit}}]\label{linf}
For all $u\in H^{2}\left(\Omega\right)$,
\begin{align*}
\|u\|_{L^{\infty}}\le c_{0}\|u\|+c_{1}\|\nabla u\|+c_{2}\|u_{xx}\|
\end{align*}
with
\begin{align*}
c_{j}=\displaystyle \frac{\gamma_{j}}{\left|\overline{\Omega}\right|}\left[\max_{x_{0}\in\overline{\Omega}}\int_{\overline{\Omega}}|x-x_{0}|^{2j}dx\right]^{1/2},~(j=0,1,2),
\end{align*}
where $u_{xx}$ denotes the Hesse matrix of $u,\ \left|\overline{\Omega}\right|$ is the measure of $\overline{\Omega}$, and
\begin{align*}
\gamma_{0}=1,~\gamma_{1}=1.1548,~\gamma_{2}=0.22361.
\end{align*}
For $n=3$, other values of $\gamma_{0},~\gamma_{1},$ and $\gamma_{2}$ have to be chosen $($see {\rm \cite[Theorem 1]{plum1992explicit}}$)$.
\end{theo}
\begin{rem}
The norm of the Hesse matrix of $u$ is precisely defined by
\begin{align*}
\|u_{xx}\|=\sqrt{\sum_{i,j=1}^{2}\left\|\frac{\partial^{2}u}{\partial x_{i}\partial x_{j}}\right\|^{2}}.
\end{align*}
Especially when $\Omega$ is polygonal, we have $\left\|u_{xx}\right\|=\left\|\Delta u\right\|$ for all $u\in H^{2}(\Omega)\cap V$~$($see, for example, {\rm \cite{grisvard2011elliptic}}$)$.
More general domains $\Omega$ are considered in {\rm \cite[Section 6.2.7]{nakaoplumwatanabe2019numerical}}.
\end{rem}
\begin{rem}
Explicit values of each $c_{j}$ are provided for some special domains $\Omega$ in {\rm \cite{plum1992explicit,plum2001computer}}.
E.g., when $\Omega=(0,1)^{2}$, we can choose
\begin{align*}
c_{0}=\displaystyle \gamma_{0}=1,\ 
c_{1}=\sqrt{\frac{2}{3}}\gamma_{1}\leq 0.9429,
{\rm~and~}c_{2}=\frac{\gamma_{3}}{3}\sqrt{\frac{28}{5}} \leq 0.1764.
\end{align*}
\end{rem}

Applying Theorem \ref{linf}, we obtain the following corollary.
\begin{coro}\label{Linfcoro}
Let $u$ be a solution of \eqref{absproblem} satisfying \eqref{h10error} with $\hat{u}\in V$ such that $\Delta \hat{u} \in L^{2}(\Omega)$.
Moreover, let $c_{0},\ c_{1}$, and $c_{2}$ be the constants in Theorem {\rm \ref{linf}}, and $p':=2(p-1)$.
We have
\begin{align}
&\left\|u-\hat{u}\right\|_{L^{\infty}}
\leq c_{0}C_{2}\alpha+c_{1}\alpha+ \nonumber \\
&c_{2}\left\{\max\{1,2^{\frac{p'-1}{2}}\}p\alpha C_{q}\sqrt{\left\|\hat{u}\right\|_{L^{rp'}}^{p'}+\frac{\alpha^{p'}}{p'+1}C_{rp'}^{p'}}+\left\|\Delta\hat{u}+\left|\hat{u}\right|^{p-1}\hat{u}\right\|\right\}\label{coroinequ}
\end{align}
for any $q$ and $r$ satisfying $q\geq 2,\ r\geq(p-1)^{-1}$, and $2q^{-1}+r^{-1}=1$.
\end{coro}
\begin{proof}
Using Theorem \ref{linf}, we have
\begin{align*}
\left\|u-\hat{u}\right\|_{L^{\infty}}&=\alpha\left\|\omega\right\|_{L^{\infty}}\\
&\leq\alpha\left(c_{0}\left\|\omega\right\|+c_{1}\left\|\omega\right\|_{V}+c_{2}\left\|\Delta\omega\right\|\right)\\
&\leq\alpha\left(c_{0}C_{2}+c_{1}+c_{2}\left\|\Delta\omega\right\|\right).
\end{align*}
The last term $\left\|\Delta\omega\right\|$ is estimated via
\begin{align*}
\alpha\left\|\Delta\omega\right\| &=\left\|\left|\hat{u}+\alpha\omega\right|^{p-1}\left(\hat{u}+\alpha\omega\right)+\Delta\hat{u}\right\| \\
&=\left\|\left|\hat{u}+\alpha\omega\right|^{p-1}\left(\hat{u}+\alpha\omega\right)-\left|\hat{u}\right|^{p-1}\hat{u}+\left|\hat{u}\right|^{p-1}\hat{u}+\Delta\hat{u}\right\| \\
&\leq\left\|\left|\hat{u}+\alpha\omega\right|^{p-1}\left(\hat{u}+\alpha\omega\right)-\left|\hat{u}\right|^{p-1}\hat{u}\right\| +\left\|\Delta\hat{u}+\left|\hat{u}\right|^{p-1}\hat{u}\right\|.
\end{align*}
The mean value theorem ensures that
\begin{align*}
&\displaystyle \int_{\Omega}\left(\left|\hat{u}(x)+\alpha\omega(x)\right|^{p-1}\left(\hat{u}(x)+\alpha\omega(x)\right)-\left|\hat{u}(x)\right|^{p-1}\hat{u}(x)\right)^{2}dx\\
&=\displaystyle \int_{\Omega}\left(p\alpha \omega(x)\int_{0}^{1}\left|\hat{u}(x)+\alpha t\omega(x)\right|^{p-1}dt\right)^{2}dx\\
&\displaystyle \leq p^{2}\alpha^{2}\int_{\Omega}\omega(x)^{2}\int_{0}^{1}\left|\hat{u}(x)+\alpha t\omega(x)\right|^{p'}dtdx\\
&=p^{2}\displaystyle \alpha^{2}\int_{0}^{1}\int_{\Omega}\omega(x)^{2}\left|\hat{u}(x)+\alpha t\omega(x)\right|^{p'}dxdt\\
&\displaystyle \leq p^{2}\alpha^{2}\left\|\omega\right\|_{L^{q}}^{2}\int_{0}^{1}\left\|\left|\hat{u}+\alpha t\omega \right|^{p'}\right\|_{L^{r}}dt\\
&=p^{2}\displaystyle \alpha^{2}\left\|\omega\right\|_{L^{q}}^{2}\int_{0}^{1}\left\|\hat{u}+\alpha t\omega \right\|_{L^{rp'}}^{p'}dt\\
&\displaystyle \leq p^{2}\alpha^{2}\left\|\omega\right\|_{L^{q}}^{2}\int_{0}^{1}\left(\left\|\hat{u}\right\|_{L^{rp'}}+\alpha t\left\|\omega\right\|_{L^{rp'}}\right)^{p'}dt\\
&\displaystyle \leq\max\{1,2^{p'-1}\}p^{2}\alpha^{2}\left\|\omega\right\|_{L^{q}}^{2}\left\{\left\|\hat{u}\right\|_{L^{rp'}}^{p'}+\int_{0}^{1}\left(\alpha t\left\|\omega\right\|_{L^{rp'}}\right)^{p'}dt\right\}\\
&=\displaystyle \max\{1,2^{p'-1}\}p^{2}\alpha^{2}\left\|\omega\right\|_{L^{q}}^{2}\left(\left\|\hat{u}\right\|_{L^{rp'}}^{p'}+\frac{\alpha^{p'}}{p'+1}\left\|\omega\right\|_{L^{rp'}}^{p'}\right)\\
&\displaystyle \leq\max\{1,2^{p'-1}\}p^{2}\alpha^{2}C_{q}^{2}\left(\left\|\hat{u}\right\|_{L^{rp'}}^{p'}+\frac{\alpha^{p'}}{p'+1}C_{rp'}^{p'}\right).
\end{align*}
Therefore, it follows that
\begin{align*}
\alpha\left\|\Delta\omega\right\|
\leq\max\{1,2^{\frac{p'-1}{2}}\}p\alpha C_{q}\sqrt{\left\|\hat{u}\right\|_{L^{rp'}}^{p'}+\frac{\alpha^{p'}}{p'+1}C_{rp'}^{p'}}+\left\|\Delta\hat{u}+\left|\hat{u}\right|^{p-1}\hat{u}\right\| .
\end{align*}
As a result, the $L^{\infty}$-error of $u$ is estimated as asserted in \eqref{coroinequ}.
\end{proof}

\begin{rem}
	The right-hand side of \eqref{coroinequ} is small when the residual $\|\Delta\hat{u}+\left|\hat{u}\right|^{p-1}\hat{u}\|$ and also the $H^1_0$-error bound $\alpha$ are small.
\end{rem}

\section{Rigorous integration method}\label{integration}
To apply Theorem \ref{theo:plum2001} to problem \eqref{absproblem}, we have to construct a ``good'' approximate solution $\hat{u}\in V$ of \eqref{absproblem} such that $\delta$ in \eqref{zansa} is sufficiently small.
In this section, we assume that such an approximation $\hat{u}$ is constructed by a finite linear combination of basis functions $\left\{\phi_{i}\right\}_{i=1}^{\infty}$ that span $V$, and each $\phi_{i}$ is in $C^{\infty}(\overline{K_j})$ (and therefore, $\hat{u}\in C^{\infty}(\overline{K_j})$) for all $j=1,2,\cdots,M$, where $\left\{K_{i}\right\}_{j=1}^{M}$ is a rectangular or triangular ``mesh'' of $ \Omega $ that satisfies $\cup_{j}\overline{K_{j}}=\overline{\Omega}$ and ${\rm measure\,} (\cap_{j}\overline{K_{j}})=0$.
This restriction remains applicable to several numerical approximation methods for \eqref{absproblem} such as finite element methods and Fourier-Galerkin methods.
The ``mesh'' here is not necessarily reasonable for constituting $\hat{u}$, but can be set artificially finer for good integration accuracy.
Indeed, in the numerical example presented in Section \ref{sec/example}, we construct $\hat{u}$ to be smooth over the entire of $\Omega = (0,1)^2$ using a Fourier basis.
However, we divide $\Omega$ into several smaller rectangles as shown in Subsection \ref{subsec:rect} for integration purposes.

To obtain explicit bounds for $\delta$ and $K$ required in Theorem \ref{theo:plum2001}, we have to compute rigorously, in particular, $(\Delta\hat{u},\left|\hat{u}\right|^{p-1}\hat{u})$ and $(\phi_{i},\left|\hat{u}\right|^{p-1}\phi_{j})$; recall that $p\in (1,2)$ which makes this integration nontrivial.
In the following subsections, we discuss rigorous integration methods for rectangular and triangular cases.
More precisely, we propose methods for computing the integral
\begin{align*}
I=\displaystyle \int_{K_j}\left\{\eta(x,y)\right\}^{q}\xi(x,y)dxdy
\end{align*}
for $q \in (0,1)$ and $\eta,$\,$\xi\in C^{\infty}(\overline{K_j})$ ($j=1,2,\cdots,M$) in a rigorous enclosure form.
Here, we assume that $\eta>0$ in $\Omega$ and $\eta=0$ on $\partial\Omega$; indeed, we later select $\eta=\hat{u}$, an approximate solution to \eqref{absproblem}, which has these properties.
The nonnegativity of $\hat{u}$ in $\Omega$ is proved through the procedures described in Subsections \ref{subsec:rect} and \ref{subsec:tri}.

Some rigorous integration methods can be applied to such an integration, under the assumption that $\eta>0$ on the whole closure of a domain $\Omega\subset \mathbb{R}^{2}$ (see, for example, \cite{storck1993numerical}).
However, because here the derivative of $\left\{\eta(\cdot,\cdot)\right\}^{q}:\Omega\rightarrow \mathbb{R}$ is generally not bounded near the boundary $\partial\Omega$, where $\eta$ vanishes, previous methods cannot be applied in our situation.
To overcome this difficultly, we use a Taylor expansion based approach as follows.

\subsection{Rectangular mesh}\label{subsec:rect}
We consider integration over $\Omega=(0,1)^{2}$, which is used in our numerical experiment in Section \ref{sec/example}.
We first divide $\Omega$ into four sub-squares and consider the integration over $(0,0.5)^{2}$.
Integration over the three other parts can be carried out similarly after translation and rotation such that $\eta=0$ on both the left and the lower edge.
Moreover, we divide $\overline{(0,0.5)^{2}}$ into closed rectangles that are grouped into four types ($S_{1,1},\ S_{1,0},\ S_{0,1}$, and $S_{0,0}$) as in Fig.~\ref{figrectangle}.
 \begin{figure}[h]
 \begin{center}
  \includegraphics[height=100mm]{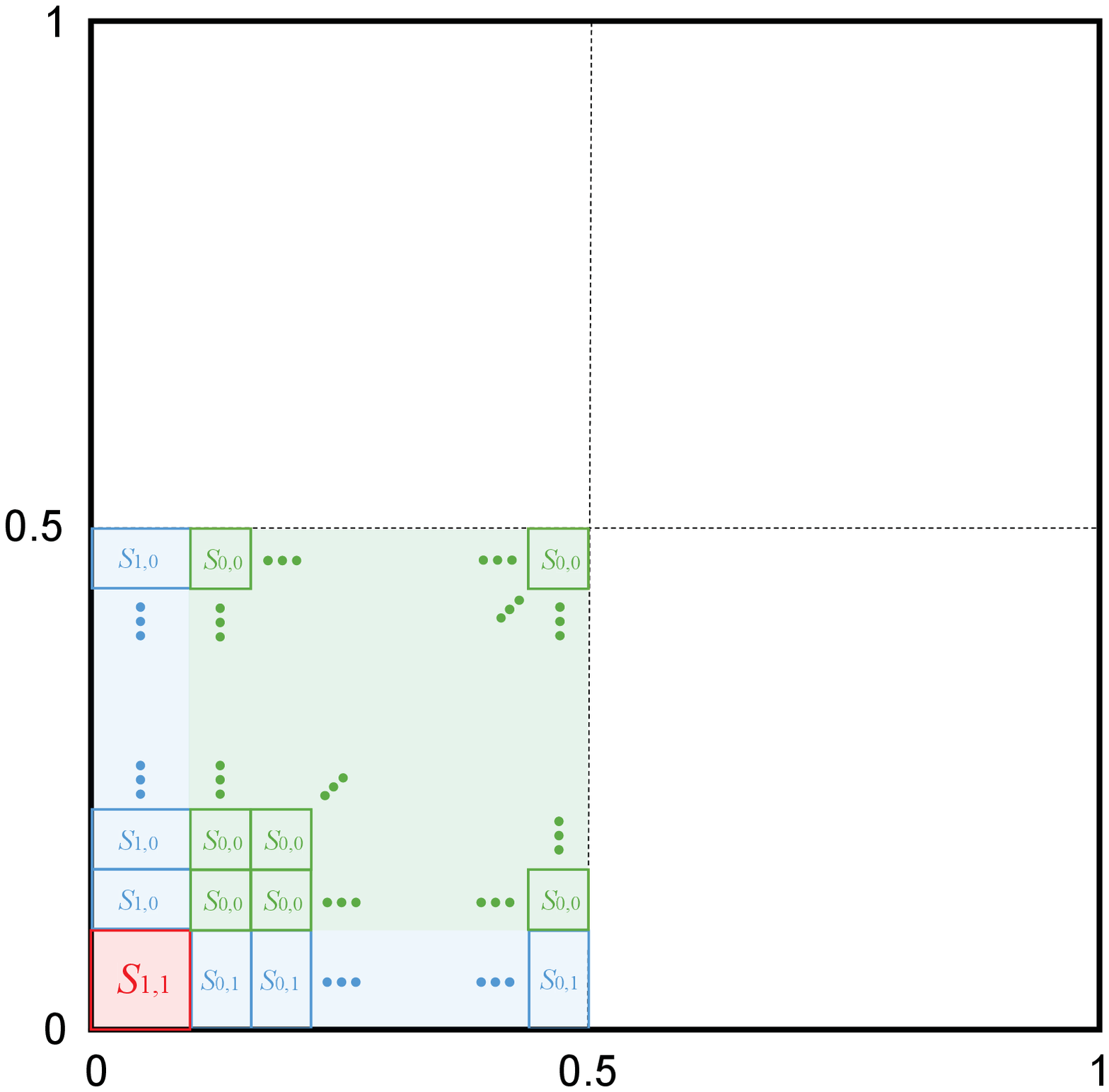}
 \end{center}
 \caption{Division of the domain $\Omega=(0,1)^{2}$.}
 \label{figrectangle} 
 \end{figure}
These types of rectangles have the following properties:
\begin{itemize}
\setlength{\leftskip}{0.5cm}
\setlength{\parskip}{0cm}
\setlength{\itemsep}{0.5pt}
\item[$S_{1,1}$:]a square where $\eta$ is zero on both the left and the lower edge;
\item[$S_{0,1}$:]rectangles where $\eta$ is zero only on the lower edge;
\item[$S_{1,0}$:]rectangles where $\eta$ is zero only on the left edge;
\item[$S_{0,0}$:]squares where $\eta>0$.
\end{itemize}
The integration over $(0,0.5)^{2}$ can be expressed by the summation of integrations over the above four types of rectangles.
Hereafter, we use the notation $\Lambda_{n}^{1,1}=\{(i,j)\in \mathbb{N}^{2}\ :\ i\leq n,\ j\leq n\}$,
$\Lambda_{n}^{0,1}=\{(i,j)\in \mathbb{N}_{0}\times \mathbb{N}\ :\ i\leq n,\ j\leq n\}$,
$\Lambda_{n}^{1,0}=\{(i,j)\in \mathbb{N}\times \mathbb{N}_{0}\ :\ i\leq n,\ j\leq n\}$,
and $\Lambda_{n}=\{(i,j)\in \mathbb{N}_{0}^{2}\ :\ i\leq n,\ j\leq n\}$,
where $\mathbb{N}=\{1,2,3\cdots\}$ and $\mathbb{N}_{0}=\{0,1,2\cdots\}$.
We discuss rigorous integration methods for the four types of domains $S_{1,1}$, $S_{0,1}$, $S_{1,0}$, and $S_{0,0}$ in the following.

\subsubsection{Integration over $S_{1,1}$}\label{intS11}
Using the Taylor expansion around the lower-left corner $(0,0)$, we enclose $\eta(x,y)$ as
\begin{align}
\displaystyle \eta(x,y)\in\sum_{(i,j)\in\Lambda_{n-1}^{1,1}}a_{i,j}x^{i}y^{j}+\sum_{(i,j)\in\Lambda_{n}^{1,1}\backslash\Lambda_{n-1}^{1,1}}[\underline{a}_{i,j},\ \overline{a}_{i,j}]x^{i}y^{j},\label{taylorexpansionS11}
\end{align}
for $(x,y)\in S_{1,1}$, where $a_{i,j},\,\underline{a}_{i,j},\,\overline{a}_{i,j}\in \mathbb{R},\ \underline{a}_{i,j}\leq\overline{a}_{i,j}$.
In Appendix \ref{psa}, we discuss a numerical method (Type-II PSA) for deriving such an enclosure.
We then denote
\begin{align*}
[\displaystyle \eta_{1,1}(x,y)]:=\sum_{(i,j)\in\Lambda_{n-1}^{1,1}}a_{i,j}x^{i-1}y^{j-1}+\sum_{(i,j)\in\Lambda_{n}^{1,1}\backslash\Lambda_{n-1}^{1,1}}[\underline{a}_{i,j},\ \overline{a}_{i,j}]x^{i-1}y^{j-1},
\end{align*}
which more precisely means the set of all continuous functions $w$ over $S_{1,1}$ such that $w(x,y)\in[\eta_{1,1}(x,y)]$ for all $(x,y)\in S_{1,1}$.
Therefore, $\eta(x,y)\in xy[\eta_{1,1}(x,y)]$.

We moreover assume that $[\eta_{1,1}(x,y)]$ is positive in $S_{1,1}$ (that is, $z>0$ for all $z\in[\eta_{1,1}(x,y)],\ (x,y)$
$\in S_{1,1}$); if $S_{1,1}$ is sufficiently small and $n$ is sufficiently large, this positivity condition is expected to hold for $\eta=\hat{u}$.
In the actual computation, this condition will be numerically checked by suitable interval arithmetic techniques \cite{moore2009introduction, rump1999book, kashiwagikv}.
If the positivity of $\eta_{1,1}$ holds, we use Type-II PSA first to enclose $[\eta_{1,1}(x,y)]^{q}$, and then, in a second step, to enclose $[\eta_{1,1}(x,y)]^{q}\xi(x,y)$ as
\begin{align*}
[\displaystyle \eta_{1,1}(x,y)]^{q}\xi(x,y)\in\sum_{(i,j)\in\Lambda_{n-1}}b_{i,j}x^{i}y^{j}+\sum_{(i,j)\in\Lambda_{n}\backslash\Lambda_{n-1}}[\underline{b}_{i,j},\ \overline{b}_{i,j}]x^{i}y^{j},
\end{align*}
where $b_{i,j},\,\underline{b}_{i,j},\,\overline{b}_{i,j}\in \mathbb{R},\ \underline{b}_{i,j}\leq\overline{b}_{i,j}$.
Thus, the integration over $S_{1,1}$ is enclosed as
\begin{align}
&\displaystyle \int_{S_{1,1}}\left\{\eta(x,y)\right\}^{q}\xi(x,y)dxdy\nonumber\\
&\displaystyle \in\sum_{(i,j)\in\Lambda_{n-1}}\int_{S_{1,1}}b_{i,j}x^{i+q}y^{j+q}dxdy+\sum_{(i,j)\in\Lambda_{n}\backslash\Lambda_{n-1}}\int_{S_{1,1}}[\underline{b}_{i,j},\ \overline{b}_{i,j}]x^{i+q}y^{j+q}dxdy.\label{exintS11}
\end{align}
\begin{rem}\label{rem:integration}
We remark on integration of a set of continuous functions in \eqref{exintS11}; that is, we explain rigorous integration of
\begin{align*}
\displaystyle \int_{y_{1}}^{y_{2}}\int_{x_{1}}^{x_{2}}ax^{p}y^{q}dxdy,
\end{align*}
where generally, $a,\ p,\ q,\ x_{1},\ x_{2},\ y_{1},$ and $y_{2}$ are real intervals.
Note that $ax^{p}y^{q}$ precisely means the set of all continuous functions $w$ over $(\underline{x_{1}},\overline{x_{2}})\times(\underline{y_{1}},\overline{y_{2}})$ such that $w(x,y)\in ax^{p}y^{q}$ for all $(x,y)\in(\underline{x_{1}},\overline{x_{2}})\times(\underline{y_{1}},\overline{y_{2}})$,
where we denote $\displaystyle \underline{z}=\inf z$ and $\displaystyle \overline{z}=\sup z$ for an interval $z$.
Whereas formally the integral is simply computed as
\begin{align*}
\int_{y_{1}}^{y_{2}}\int_{x_{1}}^{x_{2}}ax^{p}y^{q}dxdy=& \frac{a}{(p+1)(q+1)}y_{2}^{q+1}x_{2}^{p+1}-\frac{a}{(p+1)(q+1)}y_{2}^{q+1}x_{1}^{p+1}\\
-&\left(\frac{a}{(p+1)(q+1)}y_{1}^{q+1}x_{2}^{p+1}-\frac{a}{(p+1)(q+1)}y_{1}^{q+1}x_{1}^{p+1}\right),
\end{align*}
one has to compute the above formula in the correct order using interval arithmetic techniques
because the distributive law does not hold in interval arithmetics.
For example, $\int_{-1}^{1}[0.8,1]$
$xdx$ is not zero but is correctly computed as
\begin{align*}
\displaystyle \int_{-1}^{1}[0.8,1]xdx=\left[[0.4,0.5]x^{2}\right]_{-1}^{1}=[0.4,0.5]-[0.4,0.5]=[-0.1,0.1].
\end{align*}
\end{rem}

\subsubsection{Integration over $S_{0,1}$ and $S_{1,0}$}\label{intS10}
Let $(x_{0},0)$ be the midpoint of the lower edge of $S_{0,1}$.
We denote $\eta^{*}(x,y):=\eta(x+x_{0},y),\ \xi^{*}(x,y):=\xi(x+x_{0},y)$, and $S_{0,1}^{*}:=S_{0,1}-(x_{0},0)$.
Because we have
\begin{align*}
\displaystyle \int_{S_{0,1}}\left\{\eta(x,y)\right\}^{q}\xi(x,y)dxdy=\int_{S_{0,1}^{*}}\left\{\eta^{*}(x,y)\right\}^{q}\xi^{*}(x,y)dxdy,
\end{align*}
we consider the right integral.

By Taylor expanding $\eta^{*}(x,y)$ around the midpoint $(0,0)$ of the lower edge of $S_{0,1}^{*}$, we enclose $\eta^{*}(x,y)$ as
\begin{align*}
\displaystyle \eta^{*}(x,y)\in\sum_{(i,j)\in\Lambda_{n-1}^{0,1}}a_{i,j}x^{i}y^{j}+\sum_{(i,j)\in\Lambda_{n}^{0,1}\backslash\Lambda_{n-1}^{0,1}}[\underline{a}_{i,j},\ \overline{a}_{i,j}]x^{i}y^{j},
\end{align*}
for $(x,y)\in S_{0,1}^{*}$, where $a_{i,j},\,\underline{a}_{i,j},\,\overline{a}_{i,j}\in \mathbb{R},\ \underline{a}_{i,j}\leq\overline{a}_{i,j}$.
We then denote
\begin{align*}
[\displaystyle \eta_{0,1}^{*}(x,y)]:=\sum_{(i,j)\in\Lambda_{n-1}^{0,1}}a_{i,j}x^{i}y^{j-1}+\sum_{(i,j)\in\Lambda_{n}^{0,1}\backslash\Lambda_{n-1}^{0,1}}[\underline{a}_{i,j},\ \overline{a}_{i,j}]x^{i}y^{j-1}
\end{align*}
(therefore, $\eta^{*}(x,y)\in y[\eta_{0,1}^{*}(x,y)]$), and again assume that $[\eta_{0,1}^{*}(x,y)]$ is positive in $S_{0,1}^{*}$.
We then enclose $[\eta_{0,1}^{*}(x,y)]^{q}\xi^{*}(x,y)$ as
\begin{align*}
[\displaystyle \eta_{0,1}^{*}(x,y)]^{q}\xi^{*}(x,y)\in\sum_{(i,j)\in\Lambda_{n-1}}b_{i,j}x^{i}y^{j}+\sum_{(i,j)\in\Lambda_{n}\backslash\Lambda_{n-1}}[\underline{b}_{i,j},\ \overline{b}_{i,j}]x^{i}y^{j},
\end{align*}
where $b_{i,j},\,\underline{b}_{i,j},\,\overline{b}_{i,j}\in \mathbb{R},\ \underline{b}_{i,j}\leq\overline{b}_{i,j}$.
Thus, we can enclose the integral over $S_{0,1}^{*}$ as
\begin{align*}
&\displaystyle \int_{S_{0,1}^{*}}\left\{\eta^{*}(x,y)\right\}^{q}\xi^{*}(x,y)dxdy\\
&\displaystyle \in\sum_{(i,j)\in\Lambda_{n-1}}\int_{S_{0,1}^{*}}b_{i,j}x^{i}y^{j+q}dxdy+\sum_{(i,j)\in\Lambda_{n}\backslash\Lambda_{n-1}}\int_{S_{0,1}^{*}}[\underline{b}_{i,j},\ \overline{b}_{i,j}]x^{i}y^{j+q}dxdy.
\end{align*}

The integration over $S_{1,0}$ is carried out similarly by exchanging the roles of the variables $x$ and $y$.

\subsubsection{Integration over $S_{0,0}$}\label{intS00}
Although integration over $S_{0,0}$ without irregularity can be realized using standard methods $($see, for example, {\rm \cite{storck1993numerical}}$)$,
we here introduce an enclosure method along with Subsubsections \ref{intS11} and \ref{intS10} for the sake of consistency.

Let $(x_{0},y_{0})$ be the center of $S_{0,0}$, and we re-define $\eta^{*}(x,y):=\eta(x+x_{0},y+y_{0}),\ \xi^{*}(x,y):=\xi(x+x_{0},y+y_{0})$, and $S_{0,0}^{*}:=S_{0,0}-(x_{0},y_{0})$.
Because we have
\begin{align*}
\displaystyle \int_{S_{0,0}}\left\{\eta(x,y)\right\}^{q}\xi(x,y)dxdy=\int_{S_{0,0}^{*}}\left\{\eta^{*}(x,y)\right\}^{q}\xi^{*}(x,y)dxdy,
\end{align*}
we consider the right integral.

By Taylor expanding $\eta^{*}(x,y)$ around the center $(0,0)$ of $S_{0,0}^{*}$, we have
\begin{align*}
\displaystyle \eta^{*}(x,y)\in[\eta_{0,0}^{*}(x,y)]:=\sum_{(i,j)\in\Lambda_{n-1}}a_{i,j}x^{i}y^{j}+\sum_{(i,j)\in\Lambda_{n}\backslash\Lambda_{n-1}}[\underline{a}_{i,j},\ \overline{a}_{i,j}]x^{i}y^{j},
\end{align*}
for $(x,y)\in S_{0,0}^{*}$, where $a_{i,j},\,\underline{a}_{i,j},\,\overline{a}_{i,j}\in \mathbb{R},\ \underline{a}_{i,j}\leq\overline{a}_{i,j}$.
Assuming that $[\eta_{0,0}^{*}(x,y)]$ is positive on $S_{0,0}^{*}$ (because $\eta>0$ on $S_{0,0}^{*}$, this property is expected to hold if $n$ is sufficiently large), we have
\begin{align*}
[\displaystyle \eta_{0,0}^{*}(x,y)]^{q}\xi^{*}(x,y)\in\sum_{(i,j)\in\Lambda_{n-1}}b_{i,j}x^{i}y^{j}+\sum_{(i,j)\in\Lambda_{n}\backslash\Lambda_{n-1}}[\underline{b}_{i,j},\ \overline{b}_{i,j}]x^{i}y^{j},
\end{align*}
where $b_{i,j},\,\underline{b}_{i,j},\,\overline{b}_{i,j}\in \mathbb{R},\ \underline{b}_{i,j}\leq\overline{b}_{i,j}$.
Thus, it follows that
\begin{align*}
&\displaystyle \int_{S_{0,0}^{*}}\left\{\eta^{*}(x,y)\right\}^{q}\xi^{*}(x,y)dxdy\\
&\displaystyle \in\sum_{(i,j)\in\Lambda_{n-1}}\int_{S_{0,0}^{*}}b_{i,j}x^{i}y^{j}dxdy+\sum_{(i,j)\in\Lambda_{n}\backslash\Lambda_{n-1}}\int_{S_{0,0}^{*}}[\underline{b}_{i,j},\ \overline{b}_{i,j}]x^{i}y^{j}dxdy.
\end{align*}

\subsection{Triangular mesh}\label{subsec:tri}
For the triangular case where $\Omega$ is divided into a union of triangles $ \cup_{i=1}^M \overline{K_i} $ without overlap, we add a further assumption to the mesh $\left\{K_{i}\right\}_{j=1}^{M}$ so that $\eta = 0$ on at most one edge of $K_{i}$.
The exceptional situation, where $\eta = 0$ on two edges, can be avoided by an additional division of the triangle; see Fig.~\ref{fig:addedge}.
Recall that the refined mesh does not need to be a triangulation required for finite element methods.
Indeed, the new ``mesh'' displayed in Fig.~\ref{fig:addedge} cannot be used for the usual finite element method because of a lack of division for the second left triangle.
However, this is applicable for the purpose of integration.

 \begin{figure}[t]
 \begin{center}
  \includegraphics[height=80mm]{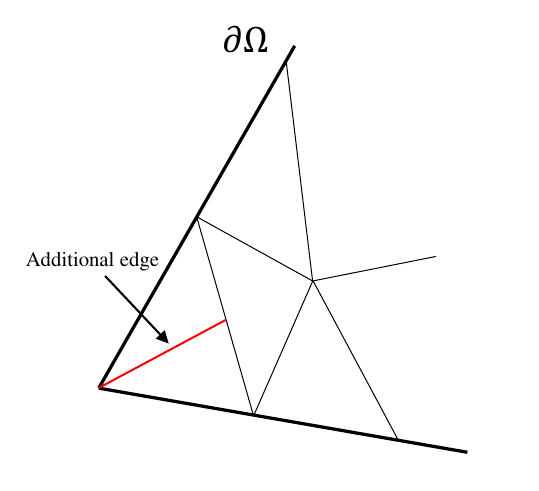}
 \end{center}
 \caption{Additional division of a triangle where $\eta = 0$ on the two edges that touch $\partial \Omega$.}
 \label{fig:addedge} 
 \end{figure}
 
 For a target triangle $K$, let $A(A_1,A_2)$, $B(B_1,B_2)$, and $C(C_1,C_2)$ be the vertices.
 If $K$ is strictly inside $\Omega$ so that no edge is shared with $\partial \Omega$ and $\eta>0$ thereon, we can use standard integration methods.
 Moreover, the method described below can be extended to such a ``regular'' situation as we did in Subsubsection \ref{intS00}.
 Therefore, we consider the case where $\eta = 0$ on the segment $\Gamma:=\overline{AB}$.
 Let $K^*$ be another triangle over new coordinates $(s,t)$ with the vertexes $A^*(-0.5,0)$, $B^*(0.5,0)$, and $C^*(0.5,1)$.
 We define the affine transformation $\phi$ for $K$ to $K^*$ by
 \begin{align}
	 \phi(A) = A^*,~~\phi(B) = B^*,~~\phi(C) = C^*
 \end{align}
 so that the center of segment $\Gamma^*:=\phi(\Gamma)$ is located at the origin of coordinates $(s,t)$; see Fig.~\ref{fig:transtri}.
 The transformation $\phi$ is explicitly given by
 \begin{align*}
 		\displaystyle \phi\left(\begin{array}{l}
 		x\\
 		y
 		\end{array}\right)=\frac{1}{D}\left(\begin{array}{l}
 		(C_{2}-B_{2})(x-A_{1})-(C_{1}-B_{1})(y-A_{2})\\
 		-(B_{2}-A_{2})(x-A_{1})+(B_{1}-A_{1})(y-A_{2})
 		\end{array}\right)
 		-
		\left(\begin{array}{l}
 		0.5\\
 		0
 		\end{array}\right)
 \end{align*}
 for $(x,y) \in K$, where $D:=(C_{2}-B_{2})(B_{1}-A_{1})-(C_{1}-B_{1})(B_{2}-A_{2})$.
 Using the transformation $\phi$, we denote $\eta^{*}(s,t):=\eta(\phi^{-1}(s,t)),\ \xi^{*}(s,t):=\xi(\phi^{-1}(s,t))$ and have 
 \begin{align*}
 \displaystyle \int_{K}\left\{\eta(x,y)\right\}^{q}\xi(x,y)dxdy = \frac{1}{|D|}\int_{K^{*}}\left\{\eta^{*}(s,t)\right\}^{q}\xi^{*}(s,t)dsdt,
 \end{align*}
 By Taylor expanding $\eta^{*}(s,t)$ around the origin $(0,0)$ on the new coordinates, we obtain the enclosure of $\eta^{*}(s,t)$ as
 \begin{align*}
 \displaystyle \eta^{*}(s,t)\in\sum_{(i,j)\in\Lambda_{n-1}^{0,1}}a_{i,j}s^{i}t^{j}+\sum_{(i,j)\in\Lambda_{n}^{0,1}\backslash\Lambda_{n-1}^{0,1}}[\underline{a}_{i,j},\ \overline{a}_{i,j}]s^{i}t^{j}.
 \end{align*}
 Therefore, we enclose the desired integration in the same manner described in Subsubsection \ref{intS10} in the following form of integration over $K^*$
 \begin{align*}
 \displaystyle \int_{0}^{1}\int_{-0.5}^{t-0.5}as^{i}t^{j}dsdt,
 \end{align*}
 where $a$ is a real interval and $i,j$ are integers.
 Specifically, this integration is computed via
  \begin{align*}
  &\displaystyle \int_{0}^{1}\int_{-0.5}^{t-0.5}as^{i}t^{j}dsdt\\
  = &\int_{0}^{1} \left\{\frac{a}{i+1}(t-0.5)^{i+1}t^j - \frac{a}{i+1}(-0.5)^{i+1}t^j\right\} dt\\
  = &\int_{0}^{1} \left\{\frac{a}{i+1}\sum_{k=0}^{i+1}\dbinom{i+1}{k} (-0.5)^{i+1-k} t^{j+k} - \frac{a}{i+1}(-0.5)^{i+1}t^j \right\} dt\\
  = &\frac{a}{i+1}\sum_{k=0}^{i+1}\frac{1}{j+k+1}\dbinom{i+1}{k} (-0.5)^{i+1-k} - \frac{a}{(i+1)(j+1)}(-0.5)^{i+1}.
  \end{align*}
  
 
 
  \begin{figure}[t]
  \begin{center}
   \includegraphics[height=60mm]{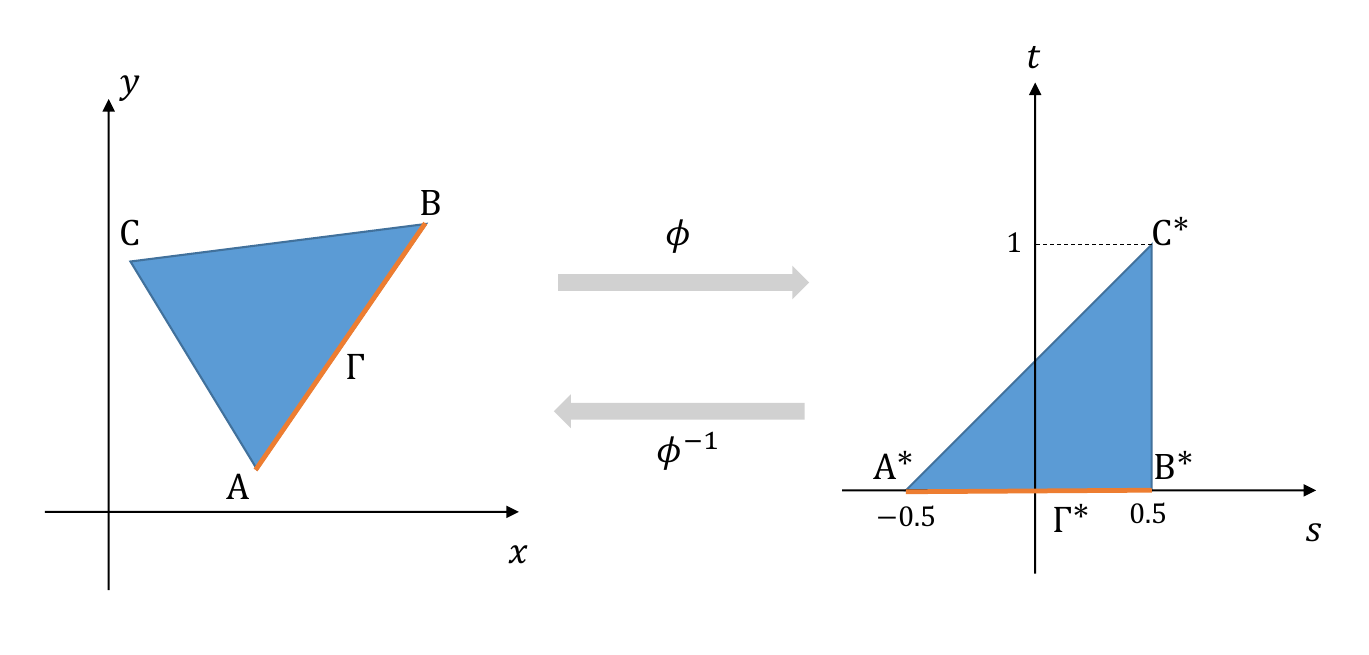}
  \end{center}
  \caption{Affine transformation $\phi$ for triangles.}
  \label{fig:transtri} 
  \end{figure}
\section{Numerical result}\label{sec/example}
In this section, we present a numerical result where a positive solution of \eqref{absproblem} is enclosed explicitly for $p=3/2$ and $\Omega=(0,1)^{2}$.
All computations were carried out on a computer with Intel Xeon E7-4830 at 2.20 GHz$\times$4, 2 TB RAM, CentOS 7, and MATLAB 2019b.
All rounding errors were strictly estimated using toolboxes for rigorous numerical computations---INTLAB version 11 \cite{rump1999book} and kv library version 0.4.49 \cite{kashiwagikv}.
Therefore, the correctness of all results was mathematically guaranteed.

We are interested in finding a reflection symmetric solution, and therefore replaced the solution space $V$ to the following subspace:
\begin{align*}
	V^0 := \left\{u\in V\ :\ u{\rm~is~symmetric~with~respect~to~}x=\frac{1}{2}{\rm~and~}y=\frac{1}{2}\right\}
\end{align*}
endowed with the same inner product as $V$.
This restriction helped us to reduce the calculation quantity somewhat.
Moreover, because eigenfunctions of \eqref{eiglam} are also restricted to symmetric functions,
eigenvalues associated with anti-symmetric eigenfunctions drop out of the minimization in \eqref{mu0}.
Therefore, $K$ can be reduced.
We calculated the other required constants $C_{p}$ and $\delta$ without exploiting such a restriction.

We selected a finite-dimensional subspace $V_{N}^0$ of $V^0$ as
\begin{align*}
V_{N}^0:=\left\{\sum_{\substack{(i,j)\in\Lambda^{1,1}_{N}\\i,j{\rm~are~odd}}}a_{i,j}\varphi_{i,j}\ :\ a_{i,j}\in \mathbb{R}\right\},
\end{align*}
where $\varphi_{i,j}(x,y)=\sin\left(i\pi x\right)\sin\left(j\pi y\right)$.
For this $V_{N}^0$, we have $C_{N}=(N+1)^{-1}\pi^{-1}$ satisfying \eqref{eq:CN-H2regular} because, for $u \in H^2(\Omega) \cap V^0$,
\begin{align*}
\left\|u-P_{N}u\right\|_{V}^{2}&=\left\|\left(u-P_{N}u\right)_{x}\right\|^{2}+\|\left(u-P_{N}u\right)_{y}\|^{2}\\
&=\displaystyle \sum_{(n,m)\in\Lambda^{1,1}_{\infty}\backslash\Lambda^{1,1}_{N}}a_{m,n}^{2}(m^{2}\pi^{2}+n^{2}\pi^{2})\left\|\varphi_{m,n}\right\|^{2}\\
&\displaystyle \leq\sum_{(n,m)\in\Lambda^{1,1}_{\infty}\backslash\Lambda^{1,1}_{N}}a_{m,n}^{2}\frac{(m^{2}\pi^{2}+n^{2}\pi^{2})^{2}}{(N+1)^{2}\pi^{2}}\left\|\varphi_{m,n}\right\|^{2}\\
&\displaystyle \leq\frac{1}{\left(N+1\right)^{2}\pi^{2}}\left\|-\Delta u\right\|^{2},
\end{align*}
where $m$ and $n$ in the summations may be restricted to odd numbers, but the value of $C_N$ remains the same.

By setting $q=r=4$ and $s=2$ in Theorem \ref{selectiong},  we obtain
\begin{align}
g\displaystyle \left(t\right)=\frac{3}{2}C_{2}^{\frac{3}{2}}C_{4}t^{\frac{1}{2}}\label{gfor3/2}
\end{align}
satisfying \eqref{lip} and \eqref{gconv} in Theorem \ref{theo:plum2001}.
Moreover, by selecting $q=4$ and $r=2$ in Corollary \ref{Linfcoro}, we have
\begin{align*}
\left\|u-\hat{u}\right\|_{L^{\infty}}
\leq c_{0}C_{2}\alpha+c_{1}\alpha+c_{2}\left\{\frac{3}{2}\alpha C_{4}\sqrt{\left\|\hat{u}\right\|+\frac{\alpha}{2}C_{2}}+\left\|\Delta\hat{u}+\left|\hat{u}\right|^{\frac{1}{2}}\hat{u}\right\|\right\}.
\end{align*}
Here, we can use the best constant $C_{2}=(\sqrt{2}\pi)^{-1}$ for the square $\Omega=(0,1)^{2}$.
Moreover, we computed $C_{4}\leq 0.318309887$ using Corollary \ref{roughboundtheo}.

An approximate solution $\hat{u}\in V^0_{N_u}$ ($N_u = 60$), which is displayed in Fig.~\ref{pic}, is computed by solving the following finite-dimensional system:
\begin{align}
	\left(\nabla \hat{u},\nabla v_{N}\right)_{L^2}
	=\left(\left|\hat{u}\right|^{\frac{1}{2}}\hat{u} ,v_{N}\right)_{L^2}~~{\rm for~all}~v_{N}\in V^{0}_{N_u} \label{eq:newtonapp}
\end{align}
via the usual Newton method with sufficient accuracy.

To rigorously compute the integrals required in the enclosing process, we divided $\Omega$ into some rectangles as displayed in Fig.~\ref{figrectangle} and used the integration method provided in Section \ref{subsec:rect}.
Using Theorem \ref{theo:plum2001}, Corollary \ref{up_positive_newkan}, and Corollary \ref{Linfcoro}, we proved the existence of a solution $u$ to \eqref{absproblem} in an $H_{0}^{1}$-ball $\overline{B}(\hat{u}, \alpha;\ \left\|\cdot\right\|_{V})$ and an $L^{\infty}$-ball $\overline{B}(\hat{u},\beta;\ \left\|\cdot\right\|_{L^{\infty}})$.
Table \ref{veri/result} shows the result of the solution enclosure.
One confirms inequality \eqref{uptheo_2_katei} and, therefore, the positivity of the enclosed solution $u$.

\begin{figure}[h]
 \begin{center}
  \includegraphics[height=80mm]{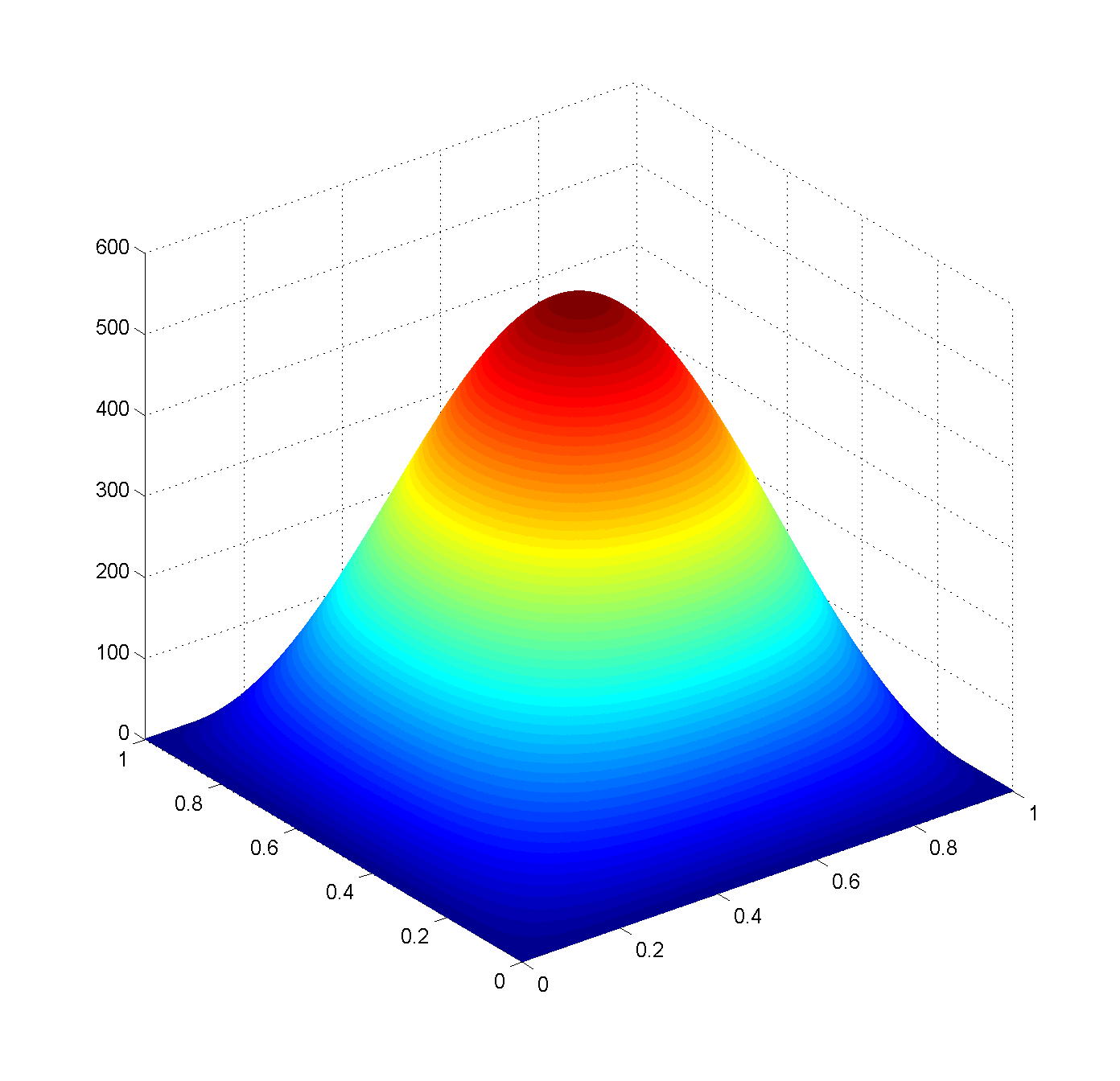}
 \end{center}
 \caption{Approximate solution of \eqref{absproblem} for $p=3/2$ on $\Omega=(0,1)^{2}$,
 the amplitude of which is proved to be in the interval [575.15, 575.61]. 
 }\label{pic}
\end{figure}

 \begin{table}[h]
  \caption{Solution enclosure result for \eqref{absproblem}.}
  \label{veri/result}
  \begin{center}
   \renewcommand\arraystretch{2}
   \footnotesize
   \begin{tabular}{cccccc}
    \hline
    $\|\Delta\hat{u}+\left|\hat{u}\right|^{\frac{1}{2}}\hat{u}\|$&
    $\delta$&
    $K\ (N=14)$&
    $\alpha$&
    $\beta$&
    $C_{\frac{5}{2}}^{2}\left(\left\|u_{-}\right\|_{L^{\frac{5}{2}}}+C_{\frac{5}{2}}\alpha\right)^{\frac{1}{2}}$\\
    \hline
    \hline
    [0.8311,~0.83150]&
    0.1871518&
    2.0000005&
    0.3909190&
    1.1462318&
    0.0227223 $(<1)$\\
    \hline
   \end{tabular}
  \end{center}
$\delta,\ K$: the constants required in Theorem \ref{theo:plum2001}.\\
The value of $\|\Delta\hat{u}+\left|\hat{u}\right|^{\frac{1}{2}}\hat{u}\|$ is proved to be in the displayed interval.
The other numerical values represent upper bounds of the corresponding constants.
\end{table}

\section{Conclusion}
	We proposed a method for proving the existence of a positive solution $u$ of Lane--Emden's equation \eqref{absproblem} close to a numerically computed approximation $\hat{u}$ together with an explicit error bound.
	In particular, we focused on the sub-square case in which $p \in (1,2)$ so that the classical Newton-Kantorovich theorem cannot be applied, and a difficulty arises in rigorous numerical integrations.
	Our method was designed based on Theorem \ref{theo:plum2001}, a generalization of the classical Newton-Kantorovich theorem, and thus well applicable in the sub-square case.
	Also, a rigorous integration method was proposed to solve the difficulty caused by the singularity.
	We presented a numerical example where an explicit solution-enclosure is obtained for $ p=3/2 $ on the unit square domain $\Omega=(0,1)^2$.
\appendix
\section{Simple bounds for embedding constants}
\def\thesection{\Alph{section}}
The following theorem provides the best constant in the classical Sobolev inequality with critical exponents.
\begin{theo}[\cite{aubin1976,talenti1976}]\label{talentitheo}
Let $u$ be any function in $W^{1,q}\left(\mathbb{R}^{n}\right)$
$(n\geq 2)$, where $q$ is any real number such that $1<q<n$.
Moreover, set $p=nq/\left(n-q\right)$.
Then, $u \in L^{p}\left(\mathbb{R}^{n}\right)$ and
\begin{align*}
\left(\int_{\mathbb{R}^{n}}\left|u(x)\right|^{p}dx\right)^{\frac{1}{p}}\leq T_{p}\left(\int_{\mathbb{R}^{n}}\left|\nabla u(x)\right|_{2}^{q}dx\right)^{\frac{1}{q}}
\end{align*}
holds for
\begin{align}
T_{p}=\pi^{-\frac{1}{2}}n^{-\frac{1}{q}}\left(\frac{q-1}{n-q}\right)^{1-\frac{1}{q}}\left\{\frac{\Gamma\left(1+\frac{n}{2}\right)\Gamma\left(n\right)}{\Gamma\left(\frac{n}{q}\right)\Gamma\left(1+n-\frac{n}{q}\right)}\right\}^{\frac{1}{n}}\label{talenticonst},
\end{align}
where
$\left|\nabla u\right|_{2}=\left((\partial u/\partial x_{1})^{2}+(\partial u/\partial x_{2})^{2}+\cdots+(\partial u/\partial x_{n})^{2}\right)^{1/2}$,
and
$\Gamma$ denotes the Gamma function.
\end{theo}

The following corollary, obtained from Theorem \ref{talentitheo}, provides a simple bound for the embedding constant from $H_{0}^{1}\left(\Omega\right)$ to $L^{p}(\Omega)$ for a bounded domain $\Omega$.

\begin{coro}[{\cite[Corollary A.2]{tanaka2017sharp}}]\label{roughboundtheo}
Let $\Omega\subset \mathbb{R}^{n}\,(n\geq 2)$ be a bounded domain.
Let $p$ be a real number such that $p\in(n/(n-1),2n/(n-2)]$ if $n\geq 3$ and $p\in(2,\infty)$ if $n=2$.
Moreover, set $q=np/(n+p).$
Then, $(\ref{embedding})$ holds for
\begin{align*}
	C_{p}=\left|\Omega\right|^{\frac{2-q}{2q}}T_{p},
\end{align*}
where $T_{p}$ is the constant in {\rm (\ref{talenticonst})}.
\end{coro}

In \cite[Lemma 7.10]{nakaoplumwatanabe2019numerical}, another formula is provided to obtain an explicit bound $C_p$.

\section{Power series arithmetic}\label{psa}
Two types of Power Series Arithmetic (called Type-I PSA and Type-II PSA) 
have been packaged in \cite{kashiwagikv}.
Although describing PSAs in detail increases this paper's capacity, we explain here them to guarantee the self-containedness of this paper.
The other reason is that there are few detailed English documents about PSAs to which can be referred.
Both PSAs were originally designed to perform operations for sets of continuous functions defined on a closed interval $D=[\underline{d},\overline{d}]$ with $\underline{d},\,\overline{d}\in \mathbb{R}$, written in the form
\begin{align}
[u(x)]=\displaystyle \sum_{i=0}^{n}u_{i}x^{i}:=\left\{v\in C(D)\ :\ v(x)\in\sum_{i=0}^{n}u_{i}x^{i}~~\forall x\in D\right\},\label{psaform}
\end{align}
where each $u_{i}$ ($i=0,1,2,\cdots,n$) is a real number or a real interval $[\underline{u_{i}},\overline{u_{i}}],\ \underline{u_{i}}\leq\overline{u_{i}}$.
Type-I PSA performs such operations with neglecting terms of degree higher than $n$.
Therefore, Type-I PSA gives approximate results of the operations.
On the other hand, Type-II PSA gives a rigorous result of such operations; that is, an operation result from Type-II PSA always includes the correct operation result in a strict mathematical sense.

In the following, we introduce the original Type-II PSA in the one-dimensional case.
Subsequently, we present a generalization of Type-II PSA to higher-dimensional cases to obtain rigorous enclosures such as \eqref{taylorexpansionS11}.
\subsection{Type-II PSA in the one-dimensional case}
We consider rigorous operations for a set of continuous functions written in the form \eqref{psaform}.
The addition and the subtraction are respectively performed as
\begin{align*}
[u(x)]+[v(x)]=\displaystyle \sum_{i=0}^{n}(u_{i}+v_{i})x^{i}
\shortintertext{and}
[u(x)]-[v(x)]=\displaystyle \sum_{i=0}^{n}(u_{i}-v_{i})x^{i}.
\end{align*}

The multiplication is performed as follows.
We first multiply $[u(x)]$ and $[v(x)]$ without degree omissions:
\begin{align*}
[u(x)]\displaystyle \times[v(x)]=\sum_{i=0}^{2n}w_{i}x^{i},~~w_{k}=\displaystyle \sum_{i=\max(0,k-n)}^{\min(k,n)}u_{i}v_{k-i}.
\end{align*}
Then, we reduce its degree from $2n$ to $n$ on the basis of the degree reduction defined as follows.
\begin{defi}[Degree reduction]\label{degred}
For a power series $[u(x)]=u_{0}+u_{1}x+\cdots+u_{m}x^{m}$ over $D$, the degree reduction $[v(x)]$ to $n$~$(n<m)$ is defined by
\begin{align*}
[v(x)]=\displaystyle \sum_{i=0}^{n}v_{i}x^{i},
\end{align*}
where
\begin{align*}
v_{i}=u_{i}~(i=0,1,\cdots,n-1)~~{\rm and}~~v_{n}=\left\{\sum_{i=n}^{m}u_{i}x^{i-n}\ :\ x\in D\right\}.
\end{align*}
\end{defi}
Thus, the terms of degree more than $n$ are resorbed in the term of degree $n$.
As a result, the multiplication by Type-II PSA includes the correct multiplication result.
\begin{rem}\label{horner}
When computing
\begin{align*}
\left\{\sum_{i=n}^{m}u_{i}x^{i-n}\ :\ x\in D\right\},
\end{align*}
we have to evaluate the range of the polynomial $u_{n}+u_{n+1}x+\cdots+u_{m}x^{m-n}$ and
the ordinary interval arithmetic occasionally over-estimates the range.
A more accurate method that evaluates the range, such as the Horner scheme, is required to obtain a precise multiplication result.
\end{rem}

We then apply Type-II PSA to general $C^{\infty}$-functions using the Taylor expansion with a remainder term.
For a $C^{\infty}$-function $f,\ f(u_{0}+u_{1}x+\cdots+u_{n}x^{n})$ is computed as
\begin{align}
&f(u_{0}+u_{1}x+\cdots+u_{n}x^{n})\nonumber\\
\subset&f(u_{0})+\displaystyle \sum_{i=1}^{n-1}\frac{1}{i!}f^{(i)}(u_{0})(u_{1}x+\cdots+u_{n}x^{n})^{i}\nonumber\\
&~~~+\displaystyle \frac{1}{n!}f^{(n)}\left({\rm hull}_{\,}\left( u_0, \left\{\sum_{i=0}^{n}u_{i}x^{i}\ :\ x\in D\right\}\right)\right)(u_{1}x+\cdots+u_{n}x^{n})^{n},\label{eq:hull}
\end{align}
by Taylor expanding $f$ around $u_{0}$, where hull\,($a,b$) denotes the convex hull of real numbers or real intervals $a$ and $b$.
Here, additions, subtractions, and multiplications in the above process are operated by Type-II PSA defined so far, and the expression
\begin{align*}
\left\{\sum_{i=0}^{n}u_{i}x^{i}\ :\ x\in D\right\}
\end{align*}
is similarly computed as mentioned in Remark \ref{horner}.
The division can be operated as $[u]/[v]:=[u]\times f([v])$ with $f(x)=1/x$ using the above method.
\begin{rem}
In our examples, the interval $D$ in \eqref{eq:hull} contains zero in all cases.
Indeed, in the integration procedures described in Section $\ref{integration}$, the domains $S_{0,1}$, $S_{1,0}$, and $S_{0,0}$ of integrations are translated to contain $(0,0)$ $($Type-II PSA in the two-dimensional case will be introduced in Appendix \ref{subsec:higherdimPSA}$)$.
Hence, in our examples,
\begin{align*}
{\rm hull}_{\,}\left( u_0, \left\{\sum_{i=0}^{n}u_{i}x^{i}\ :\ x\in D\right\}\right)
=
\left\{\sum_{i=0}^{n}u_{i}x^{i}\ :\ x\in D\right\}
\end{align*}
always holds in \eqref{eq:hull}.
\end{rem}

\begin{rem}
Type II-PSA is designed so that the coefficients of degree less than $n$ are not intervals but real numbers, and only the coefficient of degree $n$ is an interval.
However, to strictly guarantee all results from Type II-PSA in actual computation,
the coefficients of degree less than $n$ often become intervals that arise only from rounding errors.
\end{rem}

\subsection{Type-II PSA in the higher-dimensional cases}\label{subsec:higherdimPSA}
One-dimensional Type-II PSA is originally designed for power series that have real or real-interval coefficients.
However, the set of coefficients in Type-II PSA can be generalized to any set equipped with the four arithmetic operations.
Therefore, Type-II PSA can be generalized to two-dimensional cases by replacing its coefficients with one-dimensional power series
because the set of power series itself is endowed with the four arithmetic operations by Type-II PSA.
To be more precise, by replacing each coefficient $u_{i}$ in
\begin{align}
[u(x)]=\displaystyle \sum_{i=0}^{n}u_{i}x^{i},~~x\in D_{u}\label{psau}
\end{align}
with one-dimensional power series
\begin{align*}
[v_{i}(y)]=\displaystyle \sum_{j=0}^{n}v_{i,j}y^{j},~~y\in D_{v},
\end{align*}
we can regard $[u]$ as a two-dimensional power series
\begin{align}
[u(x,y)]=\displaystyle \sum_{i=0}^{n}[v_{i}(y)]x^{i}=\sum_{i=0}^{n}\sum_{j=0}^{n}v_{i,j}x^{i}y^{j},~~(x,y)\in D_{u}\times D_{v}.\label{w}
\end{align}
Thus, Type-II PSA for the one-dimensional case is naturally carried over to the two-dimensional case.
In the same way, Type-II PSA can be applied to higher-dimensional cases; that is, the $(n+1)$-dimensional power series with the four arithmetic operations are defined by replacing each coefficient $u_{i}$ in \eqref{psau} with an $n$-dimensional power series.


 

\section*{Acknowledgments}
This work was supported by CREST, JST Grant Number JPMJCR14D4.

\bibliographystyle{amsplain}
\bibliography{ref.bib}

\end{document}